\documentclass[preprint,12pt]{elsarticle}

\usepackage{amssymb}

\journal{Performance Evaluation}

\usepackage{graphicx}
\usepackage{tikz}
\usepackage{pgfplots}
\usetikzlibrary{external}
\usetikzlibrary{shapes,patterns,decorations.pathreplacing,positioning}
\usetikzlibrary{snakes, arrows}
\usetikzlibrary{positioning}
\usetikzlibrary{calc}
\usepackage{subfigure}
\usepackage{appendix}
\usepackage{amsmath}
\usepackage{amsfonts}
\usepackage{amsthm}

\usepackage{algorithm}
\usepackage{algorithmic}

\usepackage{enumerate}
\definecolor{dbrown}{RGB}{102,51,0}
\definecolor{dblue}{RGB}{80, 100, 150}
\definecolor{dgreen}{RGB}{50,205,50}

\theoremstyle{plain}
\newtheorem{theorem}{Theorem}[section]

\newtheorem{lemma}[theorem]{Lemma}

\theoremstyle{definition}
\newtheorem{problem}{Problem}[section]
\newtheorem{example}{Example}[section]

\newtheorem{definition}{Definition}[section]

\theoremstyle{remark}

\allowdisplaybreaks

\newcommand{\ie}{{\it i.e.},\ }

\newcommand{\FF}{\mathcal{F}}

\newcommand{\parti}{C}
\newcommand{\parind}{c}
\newcommand{\finepar}{Z}
\newcommand{\fineparind}{z}
\newcommand{\minpar}{W}

\newcommand{\flowc}{\phi}

\begin{document}

\begin{frontmatter}
	


\title{A linear programming approach to Markov reward error bounds for queueing networks}

\author{Xinwei Bai, Jasper Goseling}

\address{Stochastic Operations Research, Faculty of Electrical Engineering, Mathematics and Computer Science, University of Twente, P.O. Box 217, 7500 AE Enschede, The Netherlands}

%
%

\begin{abstract}
In this paper, we present a numerical framework for constructing bounds on stationary performance measures of random walks in the positive orthant using the Markov reward approach. These bounds are established in terms of stationary performance measures of a perturbed random walk whose stationary distribution is known explicitly. We consider random walks in an arbitrary number of dimensions and with a transition probability structure that is defined on an arbitrary partition of the positive orthant. Within each component of this partition the transition probabilities are homogeneous. This enables us to model queueing networks with, for instance, break-downs and finite buffers. The main contribution of this paper is that we generalize the linear programming approach of~\cite{goseling2016linear} to this class of models.
\end{abstract}


\end{frontmatter}

%
%

\section{Introduction}
\label{sec:introduction}

We present a framework for establishing bounds on stationary performance measures of a class of discrete-time random walks in the $M$-dimensional positive orthant, i.e., with state space $S=\left\{ 0,1,\dots \right\}^M$. This class of random walks enables us to model queueing networks with nodes of finite or infinite capacity, and with transition rates that depend on the number of jobs in the nodes. The latter allows us to consider, for instance, queues with break-downs or networks with overflow. The stationary performance measures that can be considered in our framework include the average number of jobs in a queue, the throughput and blocking probabilities. 

More precisely, for a random walk $R$ we assume that a unique stationary probability distribution $\pi: S \to [0,1]$ for which the balance equations hold exists, \ie there exists $\pi$ that satisfies
\begin{align}
\label{eq:balance_equation}
\pi(n)\sum_{n^\prime \in S} P(n,n^\prime) = \sum_{n^\prime \in S} \pi(n^\prime)P(n^\prime,n), \qquad \forall n\in S,
\end{align}
where $P(n,n^\prime)$ denotes the transition probability from $n$ to $n^\prime$. 
For a non-negative function $F: S\to [0,\infty)$, we are interested in the stationary performance measure,
\begin{align}
\label{eq:definition_FF}
\FF = \sum_{n = (n_1,\dots,n_M)\in S} \pi(n)F(n).
\end{align}
For example, if $F(n) = n_1$, then $\FF$ represents the average number of jobs in the first node. 

If $\pi$ is known explicitly, $\FF$ can be derived directly. However, in general it is difficult to obtain an explicit expression for the stationary probability distribution of a random walk. In this paper, we do not focus on obtaining the stationary probability distribution. Instead, our interest is in providing a general numerical framework to obtain upper and lower bounds on $\FF$ for general random walks. In line with this goal, we do not establish existence of $\FF$ a priori. Instead we will see that
if our method successfully finds an upper and lower bound, then $\FF$ exists.

Consider a perturbed random walk $\bar{R}$, of which the stationary probability distribution $\bar{\pi}$ is known explicitly. 
Moreover, we consider an $\bar{F}: S \to [0,\infty)$ for $\bar{R}$, which can be different from $F$. The bounds on $\FF$ are established in terms of
\begin{align}
\label{eq:definition_Fbar}
\bar{\FF} = \sum_{n\in S}\bar{\pi}(n)\bar{F}(n). 
\end{align}
We use the Markov reward approach, as introduced in~\cite{vandijk1988simple}, to build up these bounds. The method has been applied to various queueing networks in~\cite{vandijk1988perturbation, vandijk1998bounds, vandijk88tandem, van2004error} and an overview of this approach has been given in~\cite{vandijk11inbook}. In the works mentioned above, error bounds have been manually verified for each specific model. The verification can be quite complicated. Thus, a linear programming approach has been presented in~\cite{goseling2016linear} that provides bounds on $\FF$ for random walks in the quarter plane ($M=2$). In particular, in~\cite{goseling2016linear} the quarter plane is partitioned into four components, namely the interior, the horizontal axis, the vertical axis and the origin. Homogeneous random walks with respect to this partition, \ie transition probabilities are the same everywhere within a component, are considered there. 

In this paper, we extend the linear programming approach in~\cite{goseling2016linear}. The contribution of this paper is two-fold. First, we build up a numerical program that can be applied to general models. In~\cite{goseling2016linear}, an $R$ in the quarter plane with a specific partition is considered. The numerical program used in~\cite{goseling2016linear} cannot be easily implemented for general partitions or multi-dimensional cases. In this paper, we are able to consider an $R$ in an arbitrary dimensional state space. Moreover, we allow for general transition probability structures. For example, we can consider models such as a two-node queue with one finite buffer and one infinite buffer. We can also consider models in which the transition probabilities are dependent on the number of jobs in a queue. Secondly, in the linear programming approach in~\cite{goseling2016linear}, one important step is that we first assign values to a set of variables using their interpretation such that all the constraints hold. Next we see these variables as parameters in the problem. In this paper we formulate a linear program to obtain values for this set of variables while in~\cite{goseling2016linear} the values are manually chosen and then verified. We show that this linear program is always feasible. 

%

The problem of obtaining the stationary probability distribution has been considered in various works. For instance, methods have been developed to find $\pi$ through its probability generating function in~\cite{cohen1983boundary, fayolle1999random, resing2003tandem}. It is shown that for random walks in the quarter plane a boundary value problem can be formulated for the probability generating function. However, the boundary value problem has an explicit solution only in special cases (for example in~\cite{resing2003tandem}). If the probability generating function is obtained, the algorithm developed in~\cite{abate1992numerical} can provide a numerical inversion of the probability generating function. In addition, the matrix geometric method has been discussed in~\cite{latouche1999introduction, neuts1981matrix} for Quasi-birth-and-death (QBD) processes with finite phases, which provides an algorithmic approach to obtain the stationary probability distribution numerically. Perturbation analysis has been considered in, for example,~\cite{altman2004perturbation, heidergott2010series, heidergott2007series}, where $\pi$ is expressed in terms of the explicitly known $\bar{\pi}$. In the works mentioned above, only random walks in the two-dimensional orthant have been considered. As is mentioned above, one of our main contribution is to be able to establish performance bounds for random walks in the multi-dimensional positive orthant.

Furthermore, heavy-traffic queues have been studied in~\cite{halfin1981heavy, harrison1978diffusion, kingman1962queues, iglehart1965limiting, whitt1982heavy}, where various heavy-traffic regimes are considered and the limiting processes are given. Tail asymptotics of the stationary distribution have been considered in~\cite{miyazawa2011light}, where the existing approaches for deriving the tail asymptotics have been discussed. The tail asymptotics for specific models have also been studied in, for example,~\cite{adan2009exact, borovkov2001large, kobayashi2014tail, lieshout2008asymptotic, zwart2004exact}. Tail asymptotics of two-dimensional semi-martingale reflecting Brownian motions (SRBM) have been studied in~\cite{dai2011reflecting, dai2013stationary, harrison2009reflected}. 

The remainder of this paper is structured as follows. In Section~\ref{sec:introduction_model_description}, we define the model and notation. Then, in Section~\ref{sec:introduction_Markov_reward_approach} we review the results of the Markov reward approach. In Section~\ref{sec:extension_lp_model_problem}, we formulate optimization problems for the upper and lower bounds, which are non-convex and have countably infinite number of variables and constraints. Next, in Section~\ref{sec:extension_lp_lp_error_bound} we apply the linear programming approach and establish linear programs for the bounds.
Finally, in Section~\ref{sec:extension_lp_numerical_experiment}, we present some numerical examples.

\section{Model and notation}
\label{sec:introduction_model_description}

Let $R$ be a discrete-time random walk in $S = \left\{ 0,1,\dots \right\}^M$. 
Moreover, let $P: S \times S \to [0,1]$ be the transition probability matrix of $R$. In this paper, only transitions between the nearest neighbors are allowed, \ie $P(n,n^\prime) > 0$ only if $u\in N(n)$, where $N(n)$ denotes the set of possible transitions from $n$, \ie
\begin{align}
N(n) = \left\{ u \in \{-1,0,1\}^M \mid n+u \in S \right\}. 
\end{align}
For a finite index set $K$, we define a partition of $S$ as follows.

\begin{definition}
\label{def:partition}
$C = \left\{ C_k \right\}_{k\in K}$ is called a partition of $S$ if
\begin{enumerate}
\item $S = \cup_{k\in K} C_k$.
\item For all $j,k\in K$ and $j \neq k$, $C_j \cap C_k = \emptyset$. 
\item For any $k\in K$, $N(n) = N(n^\prime)$, $ \forall n,n^\prime \in C_k$.
\end{enumerate}
\end{definition}

The third condition, which is non-standard for a partition, ensures that all the states in a component have the same set of possible transitions. With this condition, we are able to define homogeneous transition probabilities within a component, meaning that the transition probabilities are the same everywhere in a component. Denote by $c(n)$ the index of the component of partition $C$ that $n$ is located in. We call $c:S\to K$ the index indicating function of partition $\parti$. Throughout the paper, various partitions will be used. We will use capital letters to denote partitions and the corresponding small letters to denote their component index indicating functions. 

We restrict our attention to an $R$ that is homogeneous with respect to a partition $C$ of the state space, \ie $P(n,n+u)$ depends on $n$ only through the component index $c(n)$. Therefore, we denote by $N_{c(n)}$ and $p_{c(n),u}$ the set of possible transitions from $n$ and transition probability $P(n,n+u)$, respectively. To illustrate the notation, we present the following example. 

\begin{example}
\label{ex:introduction_coarse_partition}
Consider $S = \{ 0,1,\dots \}^2$. Suppose that $C$ consists of
\begin{align*}
C_1 &= \left\{ 0 \right\} \times \left\{ 0 \right\}, \quad C_2 = \left\{ 1,2,3,4 \right\} \times \left\{ 0 \right\}, \quad C_3 = \left\{ 5,6,\dots \right\} \times \left\{ 0 \right\}, \\
C_4 &= \left\{ 0 \right\} \times \left\{ 1,2,\dots \right\}, C_5 = \left\{ 1,2,3,4 \right\} \times \left\{ 1,2,\dots \right\}, \\
C_6 &= \left\{ 5,6,\dots \right\} \times \left\{ 1,2,\dots \right\}. 
\end{align*}
The components and their sets of possible transitions are shown in Figure~\ref{fig:introduction_coarse_partition}. 

\begin{figure}[htb!]
\centering
\begin{tikzpicture}[scale = 1]
\foreach \i in {1,...,7}
\foreach \j in {1,...,5}
\filldraw [gray] (\i, \j) circle (2pt);

\foreach \i in {1,...,7}
\filldraw [gray] (\i,0) circle (2pt);

\foreach \j in {1,...,5}
\filldraw [gray] (0, \j) circle (2pt);

\filldraw [gray] (0, 0) circle (2pt);

\draw [->, >=stealth', gray] (0,0) -- (0,6) node[left, black, thick] {$n_2$};
\draw [->, >=stealth', gray] (0,0) -- (8,0) node[below, black, thick] {$n_1$};

\draw [rounded corners] (-0.4,-0.4) rectangle (0.4,0.4) node at (0.2,-0.2) {$C_1$};
		
\draw [rounded corners] (0.6,-0.4) rectangle (4.4,0.4) node at (2.5,-0.2) {$C_2$};
	
\draw [rounded corners] (7.4,-0.4) -- (4.6,-0.4) -- (4.6,0.4) -- (7.4,0.4) node at (5.5,-0.2) {$C_3$};
		
\draw [rounded corners] (-0.4,5.4) -- (-0.4,0.6) -- (0.4,0.6) -- (0.4,5.4) node at (-0.2,2.5) {$C_4$};

\draw [rounded corners] (0.6,5.4) -- (0.6,0.6) -- (4.4,0.6) -- (4.4,5.4) node at (3,2.5) {$C_5$};

\draw [rounded corners] (7.4,0.6) -- (4.6,0.6) -- (4.6,5.4) node at (7,2.5) {$C_6$};

\draw [->, very thick] (0,0) -- ++(0.8,0);
\draw [->, very thick] (0,0) -- ++(0,0.8);
\draw [->, very thick] (0,0) -- ++(0.8,0.8);
\draw node at (-0.4,0.3) {$p_{1,u}$};
	
\draw [->, very thick] (2,0) -- ++(0.8,0);
\draw [->, very thick] (2,0) -- ++(-0.8,0);
\draw [->, very thick] (2,0) -- ++(0,0.8);
\draw [->, very thick] (2,0) -- ++(0.8,0.8);
\draw [->, very thick] (2,0) -- ++(-0.8,0.8);
\draw node at (3,0.3) {$p_{2,u}$};

\draw [->, very thick] (6,0) -- ++(0.8,0);
\draw [->, very thick] (6,0) -- ++(-0.8,0);
\draw [->, very thick] (6,0) -- ++(0,0.8);
\draw [->, very thick] (6,0) -- ++(0.8,0.8);
\draw [->, very thick] (6,0) -- ++(-0.8,0.8);
\draw node at (7,0.3) {$p_{3,u}$};

\draw [->, very thick] (0,3) -- ++(0.8,0);
\draw [->, very thick] (0,3) -- ++(0,0.8);
\draw [->, very thick] (0,3) -- ++(0,-0.8);
\draw [->, very thick] (0,3) -- ++(0.8,0.8);
\draw [->, very thick] (0,3) -- ++(0.8,-0.8);
\draw node at (-0.5,3) {$p_{4,u}$};

\draw [->, very thick] (2,3) -- ++(0.8,0);
\draw [->, very thick] (2,3) -- ++(-0.8,0);
\draw [->, very thick] (2,3) -- ++(0,0.8);
\draw [->, very thick] (2,3) -- ++(0,-0.8);
\draw [->, very thick] (2,3) -- ++(0.8,0.8);
\draw [->, very thick] (2,3) -- ++(0.8,-0.8);
\draw [->, very thick] (2,3) -- ++(-0.8,0.8);
\draw [->, very thick] (2,3) -- ++(-0.8,-0.8);
\draw node at (2.5,4) {$p_{5,u}$};

\draw [->, very thick] (6,3) -- ++(0.8,0);
\draw [->, very thick] (6,3) -- ++(-0.8,0);
\draw [->, very thick] (6,3) -- ++(0,0.8);
\draw [->, very thick] (6,3) -- ++(0,-0.8);
\draw [->, very thick] (6,3) -- ++(0.8,0.8);
\draw [->, very thick] (6,3) -- ++(0.8,-0.8);
\draw [->, very thick] (6,3) -- ++(-0.8,0.8);
\draw [->, very thick] (6,3) -- ++(-0.8,-0.8);
\draw node at (6.5,4) {$p_{6,u}$};

%
%
\end{tikzpicture}
\caption{A finite partition of $S=\left\{ 0,1,\dots \right\}^2$ and the sets of possible transitions for its components.}
\label{fig:introduction_coarse_partition}
\end{figure}
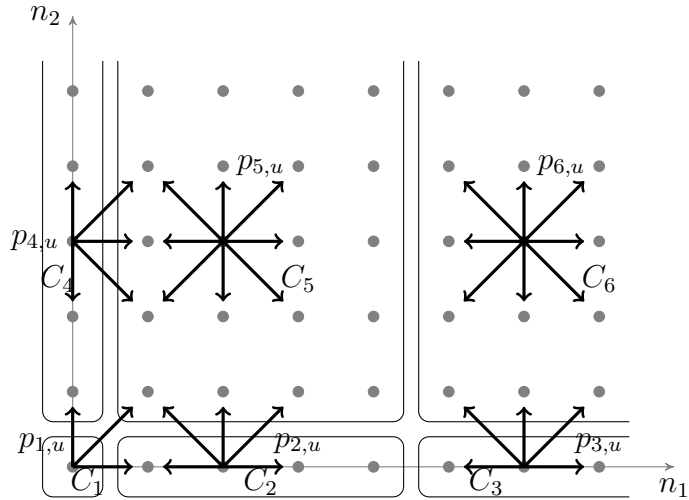
\end{example}

Based on a partition, we now define a component-wise linear function.

\begin{definition}
\label{def:c_linear_function}
Let $C$ be a partition of $S$. A function $H: S \to [0,\infty)$ is called $C$-linear if 
\begin{align}
H(n) = \sum_{k\in K} \mathbf{1}\left( n\in C_k \right)\left( h_{k,0} + \sum_{i=1}^M h_{k,i}n_i \right).
\end{align}
\end{definition}

In this paper, we often consider transformations of $H$ of the form $G(n) = H(n+u)$, $u\in N(n)$. It will be of interest to consider a partition $\finepar$ of $S$ such that $G$ is $\finepar$-linear when $H$ is $\parti$-linear. 

\begin{definition}
\label{def:fine_partition}
Given a finite partition $C$, $\finepar = \left\{ \finepar_j \right\}_{j\in J}$ is called a refinement of $C$ if
\begin{enumerate}
\item $\finepar$ is a finite partition of $S$.
\item For any $j\in J$, any $n\in \finepar_j$ and any $u\in N_j$, $c(n+u)$ depends only on $j$ and $u$, \ie
\begin{align}
c(n+u) = c(n^\prime+u), \quad \forall n,n^\prime \in \finepar_j. 
\end{align}
\end{enumerate}
\end{definition} 

Remark that a refinement of $C$ is not unique. To give more intuition, in the following example we give a refinement of $C$ that is given in Example~\ref{ex:introduction_coarse_partition}.

\begin{example}
\label{ex:introduction_fine_partition}	
In this example, consider the partition $C$ given in Example~\ref{ex:introduction_coarse_partition}. A refinement of $C$ is shown in Figure~\ref{fig:introduction_example_2d_fine}. 
	
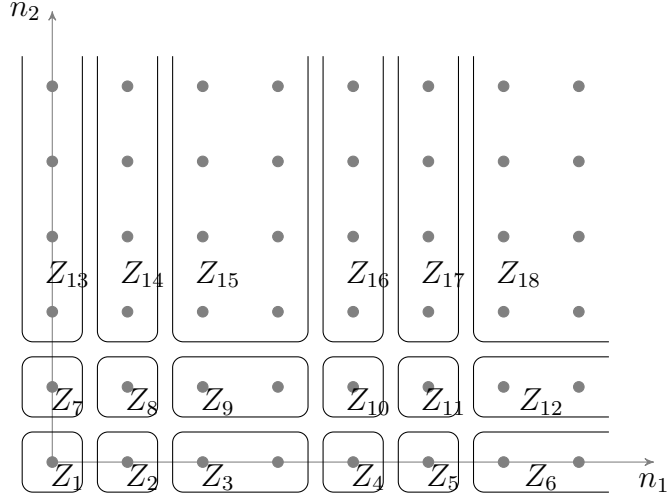
\begin{figure}[h!]
\centering
\begin{tikzpicture}[scale = 1]
		
\foreach \i in {1,...,7}
\foreach \j in {1,...,5}
\filldraw [gray] (\i, \j) circle (2pt);

\foreach \i in {1,...,7}
\filldraw [gray] (\i,0) circle (2pt);
	
\foreach \j in {1,...,5}
\filldraw [gray] (0, \j) circle (2pt);
	
\filldraw [gray] (0, 0) circle (2pt);

\draw [->, >=stealth', gray] (0,0) -- (0,6) node[left, black, thick] {$n_2$};
\draw [->, >=stealth', gray] (0,0) -- (8,0) node[below, black, thick] {$n_1$};
		
\draw [rounded corners] (-0.4,-0.4) rectangle (0.4,0.4) node at (0.2,-0.2) {$\finepar_1$};

\draw [rounded corners] (0.6,-0.4) rectangle (1.4,0.4) node at (1.2,-0.2) {$\finepar_2$};

\draw [rounded corners] (1.6,-0.4) rectangle (3.4,0.4) node at (2.2,-0.2) {$\finepar_3$};

\draw [rounded corners] (3.6,-0.4) rectangle (4.4,0.4) node at (4.2,-0.2) {$\finepar_4$};

\draw [rounded corners] (4.6,-0.4) rectangle (5.4,0.4) node at (5.2,-0.2) {$\finepar_5$};

\draw [rounded corners] (7.4,-0.4) -- (5.6,-0.4) -- (5.6,0.4) -- (7.4,0.4) node at (6.5,-0.2) {$\finepar_6$};
		
\draw [rounded corners] (-0.4,0.6) rectangle (0.4,1.4) node at (0.2,0.8) {$\finepar_7$};
		
\draw [rounded corners] (0.6,0.6) rectangle (1.4,1.4) node at (1.2,0.8) {$\finepar_8$};
	
\draw [rounded corners] (1.6,0.6) rectangle (3.4,1.4) node at (2.2,0.8) {$\finepar_9$};

\draw [rounded corners] (3.6,0.6) rectangle (4.4,1.4) node at (4.2,0.8) {$\finepar_{10}$};

\draw [rounded corners] (4.6,0.6) rectangle (5.4,1.4) node at (5.2,0.8) {$\finepar_{11}$};

\draw [rounded corners] (7.4,0.6) -- (5.6,0.6) -- (5.6,1.4) -- (7.4,1.4) node at (6.5,0.8) {$\finepar_{12}$};
		
\draw [rounded corners] (-0.4,5.4) -- (-0.4,1.6) -- (0.4,1.6) -- (0.4,5.4) node at (0.2,2.5) {$\finepar_{13}$};
		
\draw [rounded corners] (0.6,5.4) -- (0.6,1.6) -- (1.4,1.6) -- (1.4,5.4) node at (1.2,2.5) {$\finepar_{14}$};
		
\draw [rounded corners] (1.6,5.4) -- (1.6,1.6) -- (3.4,1.6) -- (3.4,5.4) node at (2.2,2.5) {$\finepar_{15}$};

\draw [rounded corners] (3.6,5.4) -- (3.6,1.6) -- (4.4,1.6) -- (4.4,5.4) node at (4.2,2.5) { $\finepar_{16}$};
		
\draw [rounded corners] (4.6,5.4) -- (4.6,1.6) -- (5.4,1.6) -- (5.4,5.4) node at (5.2,2.5) {$\finepar_{17}$};
		
\draw [rounded corners] (5.6,5.4) -- (5.6,1.6) -- (7.4,1.6) node at (6.2,2.5) {$\finepar_{18}$};
\end{tikzpicture}
\caption{A refinement of $C$ that is in Example~\ref{ex:introduction_coarse_partition}.}
\label{fig:introduction_example_2d_fine}
\end{figure}
\end{example}

Since $R$ is homogeneous with respect to partition $C$, it is homogeneous with respect to partition $Z$ as well. Next, we present the result that $H(n+u)$ is $\finepar$-linear if $H$ is $C$-linear. The proof of the lemma is straightforward and is hence omitted. 

\begin{lemma}
\label{lem:introduction_component_linear_relation}
Let $H:S \to [0,\infty)$ be a $C$-linear function. Moreover, let $\finepar$ be a refinement of $C$. For any $u \in \left\{ -1,0,1 \right\}^M$, define $G: S\to [0,\infty)$ as $G(n) = \mathbf{1}(n+u\in S)H(n+u)$. Then, $G$ is $\finepar$-linear. 
\end{lemma}


\section{Preliminaries: Markov reward approach}
\label{sec:introduction_Markov_reward_approach}

Suppose that we have obtained an $\bar{R}$ for which $\bar{\pi}$ is known explicitly. Then, we build up upper and lower bounds on $\FF$ using the Markov reward approach, an introduction to which is given in~\cite{vandijk11inbook}. In this section, we give a review of this approach including its main result. 

In the Markov reward approach, $F(n)$ is considered as a reward if $R$ stays in $n$ for one time step. Let $F^t(n)$ be the expected cumulative reward up to time $t$ if $R$ starts from $n$ at time $0$, \ie
\begin{align}
\label{eq:definition_cumulative_reward}
F^t(n) = \sum_{k=0}^{t-1}\sum_{m\in S} P^k(n,m)F(m),
\end{align}
where $P^k(n,m)$ is the $k$-step transition probability from $n$ to $m$. Then, since $R$ is ergodic and $\FF$ exists, for any $n\in S$, 
\begin{align}
\label{eq:average_gain}
\lim_{t\to \infty} \frac{F^t(n)}{t} = \FF,
\end{align}
\ie $\FF$ is the average reward gained by the random walk independent of the starting state. Moreover, based on the definition of $F^t$, it can be verified that the following recursive equation holds, 
\begin{align}
\nonumber
F^0(n) =&\ 0, \\
\label{eq:cumulative_reward}
F^{t+1}(n) =&\ F(n) + \sum_{n^\prime\in S} P(n,n^\prime)F^t(n^\prime).
\end{align}
Next, we define the bias terms as follows. 

\begin{definition}
\label{def:bias_terms}
For any $t=0,1,\dots$, the bias terms $D^t: S\times S \to \mathbb{R}$, are defined as
\begin{align}
D^t(n,n^\prime) = F^t(n^\prime) - F^t(n).
\end{align}
\end{definition}

We present the main result of the Markov reward approach below.

\begin{theorem}[Result 9.3.5 in~\cite{vandijk11inbook}]
\label{thm:Markov_reward_approach_result}
Suppose that $\bar{F}: S \to [0,\infty)$ and $G: S \to [0,\infty)$ satisfy
\begin{align}
\label{eq:error_bound_condition}
\left| \bar{F}(n) - F(n) + \sum_{n^\prime\in S} \left( \bar{P}(n,n^\prime)-P(n,n^\prime) \right) D^t(n,n^\prime) \right| \le G(n),
\end{align}
for all $n\in S$, $t \ge 0$. Then 
\begin{align*}
\left| \bar{\FF}-\FF \right| \le \sum_{n\in S}\bar{\pi}(n)G(n). 
\end{align*}
\end{theorem} 

%
%

In this paper, we obtain bounds on $\FF$ by finding $\bar{F}$ and $G$ for which~\eqref{eq:error_bound_condition} holds. 

We do not need $R$ and $\bar R$ to be irreducible. More generally, it is sufficient that there is a single absorbing communicating class (which can be different for $R$ and $\bar R$). This implies that we allow for transient states. Even though we are only interested in the steady-state behavior of our processes, it will be important for the application of the Markov reward approach to  explicitly model these transient states. The original proof of the Markov reward approach considers only irreducible processes. In~\ref{sec:introduction_proof_MRA} we provide a proof of the extension of the approach to processes with transient states.
We will use this extended result in a numerical example in Section~\ref{sec:extension_lp_numerical_experiment}. 

In addition to the bound on $\left| \bar{\FF}-\FF \right|$, the following theorem is given in~\cite{vandijk11inbook} as well, which is called the comparison result and can sometimes provide a better upper bound. 

\begin{theorem}[Result 9.3.2 in~\cite{vandijk11inbook}]
\label{thm:comparison_upper}
Suppose that $\bar{F}: S \to [0,\infty)$ satisfies
\begin{align}
\label{eq:comparison_upper_condition}
\bar{F}(n) - F(n) + \sum_{n^\prime\in S} \left( \bar{P}(n,n^\prime)-P(n,n^\prime) \right) D^t(n,n^\prime) \ge 0,
\end{align}
for all $n\in S$, $t \ge 0$. Then, 
\begin{align*}
\FF \le \bar{\FF}. 
\end{align*}
\end{theorem} 

Similarly, if the LHS of~\eqref{eq:comparison_upper_condition} is non-positive, then $\FF \ge \bar{\FF}$.

\section{Problem formulation}
\label{sec:extension_lp_model_problem}


Recall that $P(n,n^\prime)$ and $\bar{P}(n,n^\prime)$ denote the transition probability of $R$ and $\bar{R}$, respectively. Let $\Delta(n,n^\prime) = \bar{P}(n,n^\prime) - P(n,n^\prime)$. 
From the result of Theorem~\ref{thm:Markov_reward_approach_result}, the following optimization problem comes up naturally to provide an upper bound on $\FF$. 

\begin{problem}[Upper bound] 
\label{pr:extension_lp_minstart}
\begin{align}
\nonumber
\textrm{min}\ &\sum_{n\in S}\left[\bar F(n) + G(n)\right]\bar{\pi}(n), \\
\label{eq:extension_lp_basic_constraint}
\textrm{s.t.}\ &\left| \bar{F}(n) - F(n) +  \sum_{n^\prime\in S} \Delta(n,n^\prime)D^t(n,n^\prime) \right| \le G(n),\quad \forall n\in S, t\geq 0,  \\
\nonumber
&\bar F(n)\ge 0, G(n)\ge 0, \quad \forall n\in S. 
\end{align}
\end{problem}
In this problem, $\bar{F}(n)$, $G(n)$ and $D^t_u(n)$ are variables and $\bar{\pi}(n)$, $\Delta(n,n^\prime)$ are parameters. Similarly, the following optimization problem gives a lower bound on $\FF$. 

\begin{problem}[Lower bound] 
\label{pr:extension_lp_maxstart}
\begin{align}
\nonumber
\textrm{max}\ &\sum_{n\in S}\left[\bar F(n) - G(n)\right]\bar{\pi}(n), \\
\textrm{s.t.}\ &\left| \bar{F}(n) - F(n) +  \sum_{n^\prime\in S} \Delta(n,n^\prime)D^t(n,n^\prime) \right| \le G(n),\quad \forall n\in S, t\geq 0,  \\
\nonumber
&\bar F(n)\ge 0, G(n)\ge 0, \quad \forall n\in S. 
\end{align}
\end{problem}

In addition, the following problems provide a direct upper or lower bound on $\FF$, which follows from the comparison result introduced in Section~\ref{sec:introduction_Markov_reward_approach}. 

\begin{problem}[Comparison upper bound] 
\label{pr:extension_lp_comparison_upper}
\begin{align}
\nonumber
\textrm{min}\ &\sum_{n\in S} \bar F(n)\bar{\pi}(n), \\
\textrm{s.t.}\ & \bar{F}(n) - F(n) +  \sum_{n^\prime\in S} \Delta(n,n^\prime)D^t(n,n^\prime) \ge 0,\quad \forall n\in S, t\geq 0,  \\
\nonumber
&\bar F(n)\ge 0, \quad \forall n\in S. 
\end{align}
\end{problem}

\begin{problem}[Comparison lower bound] 
\label{pr:extension_lp_comparison_lower}
\begin{align}
\nonumber
\textrm{max}\ &\sum_{n\in S} \bar F(n)\bar{\pi}(n), \\
\textrm{s.t.}\ & \bar{F}(n) - F(n) +  \sum_{n^\prime\in S} \Delta(n,n^\prime)D^t(n,n^\prime) \le 0,\quad \forall n\in S, t\geq 0,  \\
\nonumber
&\bar F(n)\ge 0, \quad \forall n\in S. 
\end{align}
\end{problem}

It will be seen from numerical results that in some cases the comparison result can provide a better upper or lower bound than that obtained from Problem~\ref{pr:extension_lp_minstart} or~\ref{pr:extension_lp_maxstart}. In the remainder of this paper, we only consider Problem~\ref{pr:extension_lp_minstart}, since the other problems can be solved in the same fashion. There are countably infinite variables and constraints in Problem~\ref{pr:extension_lp_minstart}. 
In the next two sections, we will reduce Problem~\ref{pr:extension_lp_minstart} to a linear program with a finite number of variables and constraints.

\section{Linear programming approach to error bounds}
\label{sec:extension_lp_lp_error_bound}

In this section, we use the idea from~\cite{goseling2016linear} that we consider bounding functions on $D^t(n,n^\prime)$ which are independent of $t$. Replacing $D^t(n,n^\prime)$ with these bounding functions in~\eqref{eq:extension_lp_basic_constraint}, we get rid of $t$ in the constraints and obtain sufficient conditions for~\eqref{eq:extension_lp_basic_constraint}. Simultaneously, we add several extra constraints to ensure that these newly introduced functions are indeed upper and lower bounds on $D^t(n,n^\prime)$. 

In~\eqref{eq:error_bound_condition}, since only transitions between the nearest neighbors are allowed, we have $\Delta(n,n+u) = 0$ for $u\notin N_{\parind(n)}$. Then, $\Delta(n,n^\prime) D^t(n,n^\prime)$ vanishes from~\eqref{eq:error_bound_condition} for all $ n^\prime - n \notin N_{\parind(n)}$. Thus, it is sufficient to only consider the bias terms between nearest neighbors, \ie $D^t(n,n+u)$, $u\in N_{\parind(n)}$. 

More precisely, consider functions $A: S\times S\rightarrow [0,\infty)$ and $B: S\times S\rightarrow [0,\infty)$, for which
\begin{align}
\label{eq:extenstion_lp_bounds_bias}
 -A(n,n+u) \leq D^t(n,n+u) \leq B(n,n+u),
\end{align} 
for all $t\geq 0$. Then, in Problem~\ref{pr:extension_lp_minstart}, replacing $D^t(n,n^\prime)$ with the bounding functions, we get rid of the time-dependent terms and obtain the following constraints that guarantee~\eqref{eq:extension_lp_basic_constraint},
\begin{align}
\nonumber
& \bar{F}(n) - F(n) + \sum_{u\in N_{\parind(n)}} \max\left\{\Delta(n,n+u)B(n,n+u), -\Delta(n,n+u)A(n,n+u)\right\} \\
& \qquad\qquad \le G(n), \\
\nonumber
& F(n) - \bar{F}(n) + \sum_{u\in N_{\parind(n)}} \max\left\{\Delta(n,n+u)A(n,n+u), -\Delta(n,n+u)B(n,n+u)\right\} \\
& \qquad\qquad \le G(n). 
\end{align}
Besides the constraints given above, additional constraints are necessary to guarantee that~\eqref{eq:extenstion_lp_bounds_bias} holds. In the next part, we establish these additional constraints. 

Recall that $D^t(n,n+u) = F^t(n+u) - F^t(n)$. We will show in the next section that $D^{t+1}(n,n+u)$ can be expressed as a linear combination of $D^t(m,m+v)$ where $v \in N_{\parind(m)}$. More precisely, there exists $\flowc(n,u,m,v) \ge 0$ for which the following equation holds, 
\begin{align}
\label{eq:extension_lp_bias_terms_recursive}
D^{t+1}(n,n+u) = F(n+u) - F(n) + \sum_{m\in S}\sum_{v\in N_{\parind(m)}} \flowc(n,u,m,v) D^t(m,m+v),
\end{align}
for $t\ge 0$. We will reduce the sum in the equation above to a sum over a finite number of states. Therefore, the convergence of the sum is not an issue. 
Then, the following inequalities are sufficient conditions for $-A(n,n^\prime)$ and $B(n,n^\prime)$ to be a lower and upper bound on $D^t(n,n+u)$, respectively,
\begin{align}
F(n+u) - F(n) + \sum_{m\in S}\sum_{v\in N_{\parind(m)}} \flowc(n,u,m,v) B(m,m+v) &\le B(n,n+u), \\
F(n+u) - F(n) - \sum_{m\in S}\sum_{v\in N_{\parind(m)}} \flowc(n,u,m,v) A(m,m+v) &\ge -A(n,n+u).
\end{align}

Summarizing the discussion above, the following problem gives an upper bound on $\FF$.

\begin{problem} 
\label{pr:extension_lp_minlp}
\begin{align}
\nonumber
\textrm{min}\ & \sum_{n\in S}\left[\bar F(n) + G(n)\right]\bar{\pi}(n), \\
\nonumber
\textrm{s.t.}\ & \bar{F}(n) - F(n) + \sum_{u\in N_{\parind(n)}} \max\left\{\Delta(n,n+u)B(n,n+u), -\Delta(n,n+u)A(n,n+u)\right\} \\
\label{eq:extension_lp_error_bound_upper}
& \qquad\qquad \le G(n), \\
\nonumber
& F(n) - \bar{F}(n) + \sum_{u\in N_{\parind(n)}} \max\left\{\Delta(n,n+u)A(n,n+u), -\Delta(n,n+u)B(n,n+u)\right\} \\
\label{eq:extension_lp_error_bound_lower}
& \qquad\qquad \le G(n), \\
\label{eq:extension_lp_bias_terms_recursive_repeat}
& D^{t+1}(n,n+u) = F(n+u) - F(n) + \sum_{m\in S}\sum_{v\in N_{\parind(m)}} \flowc(n,u,m,v) D^t(m,m+v), \\
\label{eq:extension_lp_bias_term_upper}
& F(n+u) - F(n) + \sum_{m\in S}\sum_{v\in N_{\parind(m)}} \flowc(n,u,m,v) B(m,m+v) \le B(n,n+u), \\
\label{eq:extension_lp_bias_term_lower}
& F(n) - F(n+u) + \sum_{m\in S}\sum_{v\in N_{\parind(m)}} \flowc(n,u,m,v) A(m,m+v) \le A(n,n+u), \\
\nonumber
& \flowc(n,u,m,v) \ge 0, \quad \text{for } n,m \in S, u\in N_{\parind(n)}, v\in N_{\parind(m)} \\
\nonumber
& A(n,n+u)\ge 0, B(n,n+u)\ge 0, \bar F(n)\ge 0, G(n)\ge 0, \quad \text{for }n,n^\prime\in S. 
\end{align}
\end{problem}

In the problem the variables are $\flowc(n,u,m,v)$, $A(n,n+u)$, $B(n,n+u)$, $D^t(n,n+u)$, $\bar{F}(n)$, $G(n)$ and the parameters are $\bar{\pi}(n)$, $F(n)$, $\Delta(n,n+u)$. Problem~\ref{pr:extension_lp_minlp} is non-linear since there are terms such as $\flowc(n,u,m,v)A(n,n^\prime)$ and $\flowc(n,u,m,v)B(n,n^\prime)$. Therefore, we apply the approach used in~\cite{goseling2016linear}. More precisely, first we find a set of $\flowc(n,u,m,v)$, for which~\eqref{eq:extension_lp_bias_terms_recursive_repeat} holds. Then, we plug the obtained $\flowc(n,u,m,v)$ into Problem~\ref{pr:extension_lp_minlp} as parameters and remove~\eqref{eq:extension_lp_bias_terms_recursive_repeat}. 
%
As a consequence, Problem~\ref{pr:extension_lp_minlp} becomes linear. In~\cite{goseling2016linear}, the set of $\flowc(n,u,m,v)$ is obtained by manual derivation. In the following part, we formulate a linear program where the variables are $\flowc(n,u,m,v)$. In the linear program $\flowc(n,u,m,v)$ are interpreted as flows among states.

\subsection{Linear program for finding $\flowc(n,u,m,v)$}
\label{ssec:extension_lp_flow_problem}


In this section, we formulate a linear program to obtain $\flowc(n,u,m,v)$ for which~\eqref{eq:extension_lp_bias_terms_recursive_repeat} holds. For the bias terms, using~\eqref{eq:cumulative_reward}, we get
\begin{multline}
\label{eq:extension_lp_flow_problem_demand}
D^{t+1}(n,n+u) = F^{t+1}(n+u) - F^{t+1}(n) \\
= F(n+u) - F(n) + \sum_{m\in S}[P(n+u,m) - P(n,m)] F^t(m). 
\end{multline}
Thus,~\eqref{eq:extension_lp_bias_terms_recursive_repeat} holds if and only if
\begin{align}
\label{eq:extension_lp_D_equl_F}
\sum_{m\in S}\sum_{v\in N_{\parind(m)}} \flowc(n,u,m,v) D^t(m,m+v) = \sum_{m\in S}[P(n+u,m) - P(n,m)] F^t(m). 
\end{align}
Rewriting the LHS of~\eqref{eq:extension_lp_D_equl_F}, we have
\begin{multline}
\sum_{m\in S}\sum_{v\in N_{\parind(m)}} \flowc(n,u,m,m+v) D^t(m,m+v) \\
= \sum_{m\in S}\sum_{v\in N_{\parind(m)}} \flowc(n,u,m,m+v) [F^t(m+v) - F^t(m)] \\
= \sum_{m\in S} \left\{ \sum_{v\in N_{\parind(m)}} \left[\flowc(n,u,m+v,-v)-\flowc(n,u,m,v)\right] \right\} F^t(m).
\end{multline}
In comparison with the RHS of~\eqref{eq:extension_lp_D_equl_F}, we obtain the following constraint that is sufficient for~\eqref{eq:extension_lp_D_equl_F} as well as~\eqref{eq:extension_lp_bias_terms_recursive_repeat},
\begin{align}
\label{eq:extension_lp_flow_constraint}
\sum_{v\in N_{\parind(m)}} \left[\flowc(n,u,m+v,-v)-\flowc(n,u,m,v)\right] = P(n+u,m) - P(n,m), 
\end{align}
for all $n,m\in S$, $u\in N_{\parind(n)}$. Intuitively, for a fixed $n\in S$ and a fixed $u\in N_{\parind(n)}$, $\flowc(n,u,m,v)$ can be interpreted as a flow from state $m$ to state $m+v$, and $P(n+u,m) - P(n,m)$ can be seen as the demand at state $m$. Then, intuitively~\eqref{eq:extension_lp_flow_constraint} means that the demand at every state $m$ is equal to the difference between the inflow and outflow of $m$. 

In the next part, we formulate a linear program with a finite number of constraints and variables. Moreover, we show that based the solution of this linear program we can obtain $\flowc(n,u,m,v) \ge 0$ that satisfies~\eqref{eq:extension_lp_flow_constraint} and hence satisfies~\eqref{eq:extension_lp_bias_terms_recursive_repeat}.  The objective of this linear program is to minimize the sum of all $\flowc(n,u,m,v)$. We note that in this paper we do not optimize with respect to the overall objective, which is to find the best error bound. In the discussion section, we provide an outlook on alternative objective functions that may be used.

We need a final piece of notation. Let $\finepar = \left\{ \finepar_j \right\}_{j\in J}$ be a refinement of partition $\parti$ defined in Definition~\ref{def:fine_partition}. Then, for any $n \in Z_j$ and $u\in N_j$, let $c(j,u)$ be the index of the component of partition $C$ that $n+u$ is located in. For $j\in J$ and $u\in N_{j}$, let 
\begin{align}
\label{eq:extension_lp_definition_Nju}
N_{j,u} = N_{j} \cup \left( u+N_{c(j,u)} \right).
\end{align}
Now, we consider the following problem and present Theorem~\ref{thm:extension_lp_equiv_flow_problem}. 

\begin{problem}
\label{pr:extension_lp_flow_v1}
\begin{align}
\nonumber
\textrm{min}\quad & \quad \sum_{j\in J}\sum_{u\in N_{j}}\sum_{d\in N_{j,u}}\sum_{v\in N_{c(j,d)}} \varphi_{j,u,d,v}, \\
\nonumber
\textrm{s.t.}\ & \sum_{v\in N_{c(j,d)}} \mathbf{1}\left( d+v\in N_{j,u} \right)\left[ \varphi_{j,u,d+v,-v} - \varphi_{j,u,d,v} \right] = p_{c(j,u),d-u} - p_{j,d}, \\
\label{eq:extension_lp_flow_constraint_finite}
& \qquad\qquad \forall j\in J, u\in N_j, d\in N_{j,u}, \\
\nonumber
& \varphi_{j,u,d,v}\ge 0, \qquad \forall j \in J, u\in N_{j}, d\in N_{j,u}, v\in N_{c(j,d)}.
\end{align}
\end{problem}

\begin{theorem}
\label{thm:extension_lp_equiv_flow_problem}
Problem~\ref{pr:extension_lp_flow_v1} is feasible and has a finite number of variables and constraints. Suppose that $\varphi_{j,u,d,v}$ is the optimal solution of Problem~\ref{pr:extension_lp_flow_v1}. Then, 
\begin{align}
\label{eq:extension_lp_definition_flow}
\flowc(n,u,m,v) = 
\begin{cases}
\varphi_{\fineparind(n),u,m-n,v}, & \text{ if } m\in n+N_{\fineparind(n),u} \text{ and } m+v \in n+N_{\fineparind(n),u}, \\
0, & \text{ otherwise}, 
\end{cases}
\end{align}
satisfies~\eqref{eq:extension_lp_flow_constraint}. 
\end{theorem}
\begin{proof}
Fix some $j\in J$ and $u\in N_j$. Let $n$ be a state in $\finepar_j$. Consider an undirected graph $\mathcal{G} = (\mathcal{V}, \mathcal{E})$, where $\mathcal{V}$ contains all the nearest neighbors of $n$ and of $n+u$. Moreover, $e\in \mathcal{E}$ if and only if $e$ connects two nearest neighbors. It is easy to see that $\mathcal{G}$ is connected. From the discussion after~\eqref{eq:extension_lp_flow_constraint}, we see that~\eqref{eq:extension_lp_flow_constraint} intuitively means to find flows on $e\in \mathcal{E}$ such that the demand at every node $m\in \mathcal{V}$ is equal to the difference between the inflow and outflow of $m$. 

This is a classical flow problem in graph theory and combinatorial optimization. In our case, the graph is connected. Moreover, there is no capacity for the flows and all the demands sum up to $0$. Thus, there exists a feasible non-negative flow on $\mathcal{G}$ (see, for instance, Exercise 5 in Chapter 8 in~\cite{korte2012combinatorial}). In other words, there exists $\flowc_0(n,u,m,v) \ge 0$, where $m, m+v\in \mathcal{V}$, such that for all $m\in \mathcal{V}$,
\begin{multline}
\label{eq:feaible_flow}
\sum_{v\in N_{\parind(m)}} \mathbf{1}\left( m+v\in \mathcal{V} \right)\left[\flowc_0(n,u,m+v,-v)-\flowc_0(n,u,m,v)\right] \\ 
= P(n+u,m) - P(n,m).
\end{multline}
From~\eqref{eq:extension_lp_definition_Nju}, we see that $m\in \mathcal{V}$ if and only if $m = n+d$ for some $d\in N_{j,u}$. Take $\varphi_{j,u,d,v} = \flowc_0(n,u,n+d,v)$. Since $R$ is homogeneous with respect to partition $\parti$ as well as partition $\finepar$, $P(n+u, m) = p_{c(j,u), d-u}$ and $P(n,m) = p_{j,d}$. Therefore, we can verify that~\eqref{eq:feaible_flow} is equivalent to~\eqref{eq:extension_lp_flow_constraint_finite} hence Problem~\ref{pr:extension_lp_flow_v1} is feasible.


Suppose that $\varphi_{j,u,d,v}$ is the optimal solution of Problem~\ref{pr:extension_lp_flow_v1}. Then consider $\phi(n,u,m,v)$ where $n,m\in S$, $u\in N_{c(n)}$ and $v\in N_{c(m)}$. If $m\in n+N_{\fineparind(n),u}$ and $m+v \in n+N_{\fineparind(n),u}$, then $\varphi_{\fineparind(n),u,m-n,v}$ is well defined and satisfies~\eqref{eq:feaible_flow}. Thus, using $\varphi_{\fineparind(n),u,m-n,v} = \flowc(n,u,m,c)$ in~\eqref{eq:feaible_flow} we can verify that~\eqref{eq:extension_lp_flow_constraint} holds. Otherwise if $m\notin n+N_{\fineparind(n),u}$ or $m+v \notin n+N_{\fineparind(n),u}$,~\eqref{eq:extension_lp_flow_constraint} holds since $\flowc(n,u,m,v) = 0$ and for its RHS, $P(n+u,m) - P(n,m) = 0$. 

Finally we argue that Problem~\ref{pr:extension_lp_flow_v1} has a finite number of variables and constraints. Since there are $\left| J \right|$ components in partition $Z$ and at most $3^M$ possible transitions for every component, the number of the variables in Problem~\ref{pr:extension_lp_flow_v1} is bounded by $2\left| J \right|\cdot 27^{M}$ from above. Moreover, the number of the constraints is bounded from above by $2\left| J \right|\cdot 9^{M}$. 
\end{proof}

%

\subsection{Implementation of Problem~\ref{pr:extension_lp_minlp}}
\label{ssec:extension_lp_implementation}

Suppose that we have obtained a set of feasible coefficients $\varphi_{j,u,d,v}$ from Problem~\ref{pr:extension_lp_flow_v1}.
In this section, we show that by restricting $F(n)$, $A(n,n^\prime)$, $B(n,n^\prime)$ to be $C$-linear and using the partition structure of $S$ described in Section~\ref{sec:introduction_model_description}, Problem~\ref{pr:extension_lp_minlp} can be reduced to a linear program with a finite number of variables and constraints.

Since we only consider the bias terms between the nearest neighbors, we rewrite the bounding functions as $A_u(n)$ and $B_u(n)$ for $n\in S$ and $u\in N_{\parind(n)}$. 
Then, plugging $\varphi_{j,u,d,v}$ as parameters into Problem~\ref{pr:extension_lp_minlp} and removing~\eqref{eq:extension_lp_bias_terms_recursive_repeat}, Problem~\ref{pr:extension_lp_minlp} is equivalent to the following problem. 
\begin{problem}
\label{pr:extension_lp_min_lp_v1}
\begin{align}
\nonumber
\textrm{min}\ & \sum_{n\in S}\left[\bar F(n) + G(n)\right]\bar{\pi}(n), \\ 
\textrm{s.t.}\ & \bar{F}(n) - F(n) + \sum_{u\in N_{c(n)}} \max\left\{\Delta_{c(n),u}B_u(n), -\Delta_{c(n),u}A_u(n)\right\} - G(n) \le 0, \\ 
& F(n) - \bar{F}(n) + \sum_{u\in N_{c(n)}} \max\left\{\Delta_{c(n),u}A_u(n), -\Delta_{c(n),u}B_u(n)\right\} - G(n) \le 0, \\ 
& F(n+u) - F(n) + \sum_{d \in N_{z(n),u}} \sum_{v \in N_{c(z(n),d)}} \varphi_{z(n),u,d,v} B_v(n+d) - B_u(n) \le 0, \\ 
& F(n) - F(n+u) + \sum_{d \in N_{z(n),u}} \sum_{v \in N_{c(z(n),d)}} \varphi_{z(n),u,d,v} A_v(n+d) - A_u(n) \le 0, \\ 
\nonumber
& A_u(n)\ge 0, B_u(n)\ge 0, \bar{F}(n)\ge 0, G(n)\ge 0, \quad \text{for }n\in S, u\in N_{c(n)}. 
\end{align}
\end{problem}

In the problem the variables are $A_u(n)$, $B_u(n)$, $\bar{F}(n)$ and $G(n)$. Next, we give the reduction for Problem~\ref{pr:extension_lp_min_lp_v1} by restricting $\bar{F}$, $G(n)$, $A_u$ and $B_u$ to be $C$-linear. By Lemma~\ref{lem:introduction_component_linear_relation}, we know that $A_v(n+d)$ and $B_v(n+d)$ are $\finepar$-linear. Thus, it is easy to check that all the constraints in Problem~\ref{pr:extension_lp_min_lp_v1} have the form, 
\begin{align*}
H(n) \le 0,
\end{align*}
where $H(n)$ is $\finepar$-linear. 

For any $\finepar_j$ and $i \in \left\{ 1,\dots,M \right\}$, define $L_{j,i}$ and $U_{j,i}$ as 
\begin{align}
L_{j,i} = \min_{n\in \finepar_j} n_i, \qquad U_{j,i} = \sup_{n\in \finepar_j} n_i. 
\end{align}
Notice that $\finepar_j$ can be unbounded in dimension $i$, in which case $U_{j,i} = \infty$. Moreover, let $I(\finepar_j)$ be the set containing all the unbounded dimensions of $\finepar_j$ and $\partial\finepar_j$ be the corners of $\finepar_j$, \ie
\begin{align}
I(\finepar_j) =& \left\{ i \in \left\{ 1,2,\dots,M \right\} \mid U_{j,i} = \infty \right\}, \\
\partial\finepar_j =& \left\{ n\in \finepar_j \mid n_i = L_{j,i}, \ \forall i\in I(\finepar_j), \quad n_k \in \left\{ L_{j,k},U_{j,k} \right\}, \ \forall k\notin I(\finepar_j) \right\}.
\end{align}
Then, for the constraint $H(n) \le 0$ for $n\in \finepar_j$, sufficient and necessary conditions can be obtained in terms of the coefficients $h_{j,i}$. We give the following lemma to specify these conditions. The proof for this lemma is straightforward and hence is omitted. 

\begin{lemma}
\label{lem:extension_lp_equivalence_constraint}
Suppose that $H(n)$ is $\finepar$-linear. Then, $H(n) \le 0$ for all $n\in n\in \finepar_j$
if and only if
\begin{align}
\label{eq:extension_lp_z_linear_negative_coeff}
H(n) \le 0, \ \forall\ n\in \partial\finepar_j, \qquad h_{j,i} \le 0, \ \forall i\in J(\finepar_j).
\end{align}
\end{lemma}
%
%

For any $n\in \partial\finepar_j$, clearly $H(n) = h_{j,0}+\sum_{i=1}^M h_{j,i}n_i$ is linear in the coefficients $h_{j,i}$. For each bounded dimension, there are at two corners of $\finepar_j$. Thus,~\eqref{eq:extension_lp_z_linear_negative_coeff} contains at most $2^M$ linear constraints in $h_{j,i}$.  

Next, consider the objective function of Problem~\ref{pr:extension_lp_min_lp_v1}. In the next lemma, we show that it can be written as a linear combination of the coefficients $\bar{f}_{k,i}$ and $g_{k,i}$. The proof for the lemma is straightforward and hence is omitted. 

\begin{lemma}
\label{lem:extension_lp_equivalence_objective}
Suppose that $\bar{F}: S \rightarrow [0,\infty)$ and $G: S \rightarrow [0,\infty)$ are $\parti$-linear.
Then, 
\begin{multline}
\sum_{n\in S}\left[ \bar{F}(n)+G(n) \right]\bar{\pi}(n) = \sum_{k\in K} \left( \bar{f}_{k,0}+g_{k,0} \right) \sum_{n\in C_k}\bar{\pi}(n) \\
 + \sum_{k\in K} \sum_{i=1}^M \left( \bar{f}_{k,i}+g_{k,i} \right) \sum_{n\in C_k} n_i\bar{\pi}(n). 
\end{multline}
\end{lemma}

Therefore, based on the two lemmas above, we give the main result of this section in the following theorem. 

\begin{theorem}
Suppose that $\bar{F}$, $G$, $A_u$ and $B_u$ are $C$-linear. Then, Problem~\ref{pr:extension_lp_min_lp_v1} can be reduced to a linear program with a finite number of variables and constraints. 
\end{theorem}
\begin{proof}
From Lemmas~\ref{lem:extension_lp_equivalence_constraint} and~\ref{lem:extension_lp_equivalence_objective}, we see that Problem~\ref{pr:extension_lp_min_lp_v1} can be reduced to a linear program where the coefficients of the functions are variables. 
Next, we will show that there is a finite number of variables and constraints in the reduced problem. 

There are at most $\left| K \right|$ components and at most $3^M$ transitions from each state. Since $\bar{F}$, $G$, $A_u$ and $B_u$ are $C$-linear, the total number of coefficients is at most $2\left| K \right|(3^M+1)(M+1)$. Hence, the number of variables in Problem~\ref{pr:extension_lp_min_lp_v1} is finite. Moreover, for each component $\finepar_j$, there are at most $2^M$ corners and at most $M$ unbounded dimensions. Hence, each constraint in Problem~\ref{pr:extension_lp_min_lp_v1} can be reduced to at most $\left| J \right|(M+2^M)$ constraints. Then, the number of constraints is finite.  
\end{proof}

\section{Numerical experiments}
\label{sec:extension_lp_numerical_experiment}

In this section, we consider some numerical examples for various queueing networks and establish upper and lower bounds on various performance measures. We have used Pyomo~\cite{hart2017pyomo}, a Python-based, open-source optimization modeling language package, to implement the optimization problems. The Gurobi solver~\cite{gurobi} has been used to obtain solutions to these problems.


\subsection{Finite two-node tandem system}

%
Consider a tandem system containing two nodes. Every job arrives at Node 1 according to a Poisson process and then goes to Node 2. Each node has a capacity for jobs that can be allowed. Let $M$ and $N$ denote the capacity of Node 1 and Node 2, respectively. An arriving job is rejected and lost if Node 1 is saturated. When Node 2 is saturated, a job remains at Node 1 upon completion. Let $\lambda$ be the arrival rate. For Node 1, we consider a threshold $T \le M$. The service rate is $\mu_1$ if the number of jobs in Node 1 is no more than $T$ and $\mu_1^*$ otherwise. The service rate of Node 2 is always $\mu_2$. Assume that $\lambda < \mu_1$, $\lambda < \mu_1^*$ and $\lambda < \mu_2$. This system does not have a product-form stationary probability according to~\cite{vandijk1988simple}. 

\subsubsection*{The original random walk}

Let $n=(n_1,n_2)$ represent the number of jobs in the system. Then the state space is $S = \left\{ 0,1,\dots \right\}^2$. Note that the tandem system is a continuous-time system. We apply the uniformization technique introduced in~\cite{grassmann1977transient} to transform the system into a discrete-time random walk $R$. Without loss of generality, we assume that $\lambda + \max\left\{ \mu_1,\mu_1^* \right\} + \mu_2 \le 1$ and take uniformization constant $1$. First we describe the resulting transition probabilities for $n\in\left\{0,1,\dots,M\right\} \times \left\{0,1,\dots,N\right\}$
\begin{align}
P(n,n+e_1) &= \lambda\mathbf{1}(n_1 < M), \\ 
P(n,n-e_2) &= \mu_2\mathbf{1}(n_2 > 0), \\
P(n,n+d_1) &= 
\begin{cases}
\mu_1\mathbf{1}(n_1 > 0, n_2 < N), & n_1 \le T, \\
\mu_1^*\mathbf{1}(n_2 < N), & n_1 > T,
\end{cases} \\
P(n,n) &= 1 - \sum_{u\in \left\{ e_1,d_1,-e_2 \right\}}P(n,n+u),
\end{align}
where $e_1 = (1,0)$, $d_1 = (-1,1)$ and $e_2 = (0,1)$. 

We see that $\left\{0,1,\dots,M\right\} \times \left\{0,1,\dots,N\right\}$ forms a communicating class. Next, we define transition probabilities for the states outside $\left\{0,1,\dots,M\right\} \times \left\{0,1,\dots,N\right\}$ in such a way that these states are transient and such that $\left\{0,1,\dots,M\right\} \times \left\{0,1,\dots,N\right\}$ is absorbing.
The remaining transition probabilities are
\begin{align}
P(n,n+e_1) &= \lambda, \\ 
P(n,n-e_2) &= \mu_2\mathbf{1}(n_2 > 0), \\
P(n,n+d_1) &= 
\begin{cases}
\mu_1\mathbf{1}(n_1 > 0), & n_1 \le T, \\
\mu_1^*, & n_1 > T,
\end{cases} \\
P(n,n) &= 1 - \sum_{u\in \left\{ e_1,d_1,-e_2 \right\}}P(n,n+u),
\end{align}
The transition probabilities of $R$ are shown in Figure~\ref{fig:extension_lp_tandem2dB_original_extend}.

\begin{figure}[htb!]
\centering
\begin{tikzpicture}[scale = 0.8]
\draw [->, >=stealth'] (0,0) -- (0,8) node[above, thick] {$n_2$};
\draw [->, >=stealth'] (0,0) -- (13,0) node[right, thick] {$n_1$};
\draw [thick] (0,5) node[left]{$N$} -- (9,5);
\draw [thick] (9,0) node[below]{$M$} -- (9,5);
\draw [dashed] (4,0) node[below]{$T$} -- (4,8);

\draw [->, very thick] (3,3) -- ++(0.8,0) node[above] {\footnotesize $\lambda$};
\draw [->, very thick] (3,3) -- ++(-0.8,0.8) node[above] {\footnotesize $\mu_1$};
\draw [->, very thick] (3,3) -- ++(0,-0.8) node[left] {\footnotesize $\mu_2$}; 	

\draw [->, very thick] (3,7) -- ++(0.8,0) node[above] {\footnotesize $\lambda$};
\draw [->, very thick] (3,7) -- ++(-0.8,0.8) node[above] {\footnotesize $\mu_1$};
\draw [->, very thick] (3,7) -- ++(0,-0.8) node[left] {\footnotesize $\mu_2$};

\draw [->, very thick] (6,3) -- ++(0.8,0) node[above] {\footnotesize $\lambda$};
\draw [->, very thick] (6,3) -- ++(-0.8,0.8) node[above] {\footnotesize $\mu_1^*$};
\draw [->, very thick] (6,3) -- ++(0,-0.8) node[left] {\footnotesize $\mu_2$}; 	

\draw [->, very thick] (6,7) -- ++(0.8,0) node[above] {\footnotesize $\lambda$};
\draw [->, very thick] (6,7) -- ++(-0.8,0.8) node[above] {\footnotesize $\mu_1^*$};
\draw [->, very thick] (6,7) -- ++(0,-0.8) node[left] {\footnotesize $\mu_2$};

\draw [->, very thick] (11,7) -- ++(0.8,0) node[above] {\footnotesize $\lambda$};
\draw [->, very thick] (11,7) -- ++(-0.8,0.8) node[above] {\footnotesize $\mu_1^*$};
\draw [->, very thick] (11,7) -- ++(0,-0.8) node[left] {\footnotesize $\mu_2$};


\draw [->, very thick] (9,3) -- ++(-0.8,0.8) node[above] {\footnotesize $\mu_1^*$};
\draw [->, very thick] (9,3) -- ++(0,-0.8) node[left] {\footnotesize $\mu_2$}; 	

\draw [->, very thick] (3,0) -- ++(0.8,0) node[above] {\footnotesize $\lambda$};
\draw [->, very thick] (3,0) -- ++(-0.8,0.8) node[above] {\footnotesize $\mu_1$};


\draw [->, very thick] (6,0) -- ++(0.8,0) node[below] {\footnotesize $\lambda$};
\draw [->, very thick] (6,0) -- ++(-0.8,0.8) node[above] {\footnotesize $\mu_1^*$};

\draw [->, very thick] (9,0) -- ++(-0.8,0.8) node[above] {\footnotesize $\mu_1^*$};

\draw [->, very thick] (11,0) -- ++(0.8,0) node[above] {\footnotesize $\lambda$};
\draw [->, very thick] (11,0) -- ++(-0.8,0.8) node[above] {\footnotesize $\mu_1$};

\draw [->, very thick] (0,3) -- ++(0.8,0) node[right] {\footnotesize $\lambda$};
\draw [->, very thick] (0,3) -- ++(0,-0.8) node[left] {\footnotesize $\mu_2$};

\draw [->, very thick] (0,5) -- ++(0.8,0) node[above] {\footnotesize $\lambda$};
\draw [->, very thick] (0,5) -- ++(0,-0.8) node[left] {\footnotesize $\mu_2$};

\draw [->, very thick] (0,7) -- ++(0.8,0) node[above] {\footnotesize $\lambda$};
\draw [->, very thick] (0,7) -- ++(0,-0.8) node[left] {\footnotesize $\mu_2$};

\draw [->, very thick] (0,0) -- ++(0.8,0) node[below] {\footnotesize $\lambda$};

\draw [->, very thick] (3,5) -- ++(0.8,0) node[above] {\footnotesize $\lambda$};
\draw [->, very thick] (3,5) -- ++(0,-0.8) node[right] {\footnotesize $\mu_2$};

\draw [->, very thick] (6,5) -- ++(0.8,0) node[above] {\footnotesize $\lambda$};
\draw [->, very thick] (6,5) -- ++(0,-0.8) node[right] {\footnotesize $\mu_2$};


\draw [->, very thick] (9,5) -- ++(0,-0.8) node[right] {\footnotesize $\mu_2$};
\end{tikzpicture}
\caption{Transition probabilities of $R$. }
\label{fig:extension_lp_tandem2dB_original_extend}
\end{figure}
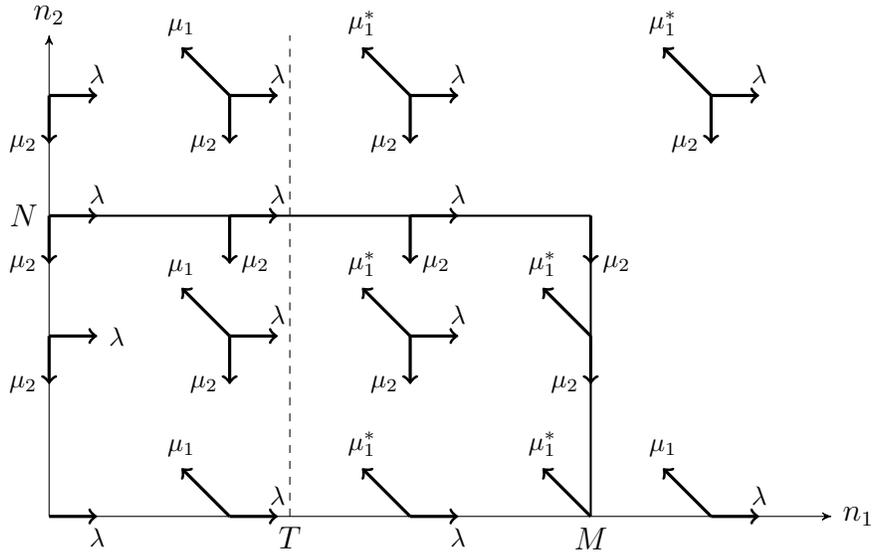


From the transition structure in Figure~\ref{fig:extension_lp_tandem2dB_original_extend}, it is clear that we can consider the partition $C$ of $S$ given in Figure~\ref{fig:extension_lp_tandem2dB_partition}. 

\begin{figure}[htb!]
\centering
\begin{tikzpicture}[scale = 1]
\foreach \i in {1,...,7}
\foreach \j in {1,...,5}
\filldraw [gray] (\i, \j) circle (2pt);

\foreach \i in {1,...,7}
\filldraw [gray] (\i,0) circle (2pt);

\foreach \j in {1,...,5}
\filldraw [gray] (0, \j) circle (2pt);

\filldraw [gray] (0, 0) circle (2pt);

\draw [->, >=stealth', gray] (0,0) -- (0,6) node[left, black, thick] {$n_2$};
\draw [->, >=stealth', gray] (0,0) -- (8,0) node[below, black, thick] {$n_1$};

\draw [rounded corners] (-0.4,-0.4) rectangle (0.4,0.4) node at (0.2,0.2) {$C_1$};

\draw [rounded corners] (0.6,-0.4) rectangle (2.4,0.4) node at (2,-0.3){$T$};
\draw node at (1.5,0.2) {$C_2$};

\draw [rounded corners] (2.6,-0.4) rectangle (4.4,0.4);
\draw node at (3.5,0.2) {$C_3$};

\draw [rounded corners] (4.6,-0.4) rectangle (5.4,0.4) node at (5,-0.3){$M$};
\draw node at (5.2,0.2) {$C_4$};

\draw [rounded corners] (-0.4,0.6) rectangle (0.4,3.4) node at (0.2,2.5) {$C_5$};

\draw [rounded corners] (0.6,0.6) rectangle (2.4,3.4) node at (1.5,2.5) {$C_6$};

\draw [rounded corners] (2.6,0.6) rectangle (4.4,3.4) node at (3.5,2.5) {$C_7$};

\draw [rounded corners] (4.6,0.6) rectangle (5.4,3.4) node at (5,2.5) {$C_8$};

\draw [rounded corners] (-0.4,3.6) rectangle (0.4,4.4) node at (0.2,3.8) {$C_9$};
\draw node at (-0.3,4){$N$};

\draw [rounded corners] (0.6,3.6) rectangle (2.4,4.4) node at (1.5,3.8) {$C_{10}$};

\draw [rounded corners] (2.6,3.6) rectangle (4.4,4.4) node at (3.5,3.8) {$C_{11}$};

\draw [rounded corners] (4.6,3.6) rectangle (5.4,4.4) node at (5.2,3.8) {$C_{12}$};

\draw [rounded corners] (7.4,-0.4) -- (5.6,-0.4) -- (5.6,0.4) -- (7.4,0.4) node at (6.5,0.2) {$C_{13}$};

\draw [rounded corners] (7.4,0.6) -- (5.6,0.6) -- (5.6,3.4) -- (7.4,3.4) node at (6.5,2.5) {$C_{14}$};

\draw [rounded corners] (7.4,3.6) -- (5.6,3.6) -- (5.6,4.4) -- (7.4,4.4) node at (6.5,3.8) {$C_{15}$};

\draw [rounded corners] (7.4,4.6) -- (5.6,4.6) -- (5.6,5.4) node at (6.5,4.8) {$C_{16}$};

\draw [rounded corners] (4.6,5.4) -- (4.6,4.6) -- (5.4,4.6) -- (5.4,5.4) node at (5.2,4.8) {$C_{17}$};

\draw [rounded corners] (2.6,5.4) -- (2.6,4.6) -- (4.4,4.6) -- (4.4,5.4) node at (3.5,4.8) {$C_{18}$};

\draw [rounded corners] (0.6,5.4) -- (0.6,4.6) -- (2.4,4.6) -- (2.4,5.4) node at (1.5,4.8) {$C_{19}$};

\draw [rounded corners] (-0.4,5.4) -- (-0.4,4.6) -- (0.4,4.6) -- (0.4,5.4) node at (0.2,4.8) {$C_{20}$};
\end{tikzpicture}
\caption{Partition $C$ of $S = \left\{ 0,1,\dots \right\}^2$. }
\label{fig:extension_lp_tandem2dB_partition}
\end{figure}
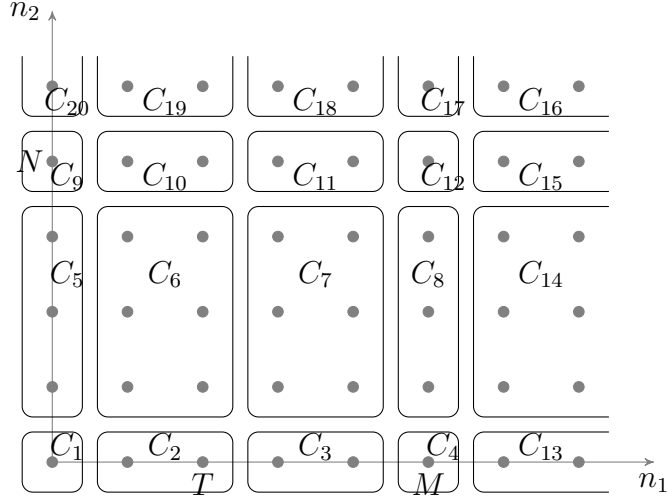

\subsubsection*{The perturbed random walk}

For the perturbed random walk, consider an $\bar{R}$ in $S$. The transition probabilities of $\bar{R}$ are 
\begin{align}
P(n,n+e_1) &= \lambda\mathbf{1}(n+e_1 \in S), \\ 
P(n,n-e_2) &= \mu_2\mathbf{1}(n-e_2 \in S), \\
P(n,n+d_1) &= 
\begin{cases}
\mu_1\mathbf{1}(n+d_1 \in S), & n_1 \le T, \\
\mu_1^*\mathbf{1}(n+d_1 \in S), & n_1 > T,
\end{cases} \\
P(n,n) &= 1 - \sum_{u\in \left\{ e_1,d_1,-e_2 \right\}}P(n,n+u). 
\end{align}
The transition probabilities of $\bar{R}$ are shown in Figure~\ref{fig:extension_lp_tandem2dB_perturb}. We can verify that the stationary probability distribution of $\bar{R}$ is 
\begin{align}
\bar{\pi}(n) = 
\begin{cases}
C\cdot \rho_1^{n_1}\sigma^{n_2}, & n_1 \le T, \\
C\cdot \rho_1^{T}\rho_2^{n_1 - T}\sigma^{n_2}, & n_1 > T, \\
\end{cases}
\end{align}
where $\rho_1 = \lambda/\mu_1$, $\rho_2 = \lambda/\mu_1^*$, $\sigma = \lambda/\mu_2$ and $C$ is the normalization constant, \ie $C^{-1} = (1-\rho_1)^{-1}(1-\rho_1^{T+1})(1-\sigma)^{-1} + \rho_1^T\rho_2(1-\rho_2)^{-1}(1-\sigma)^{-1}$. 

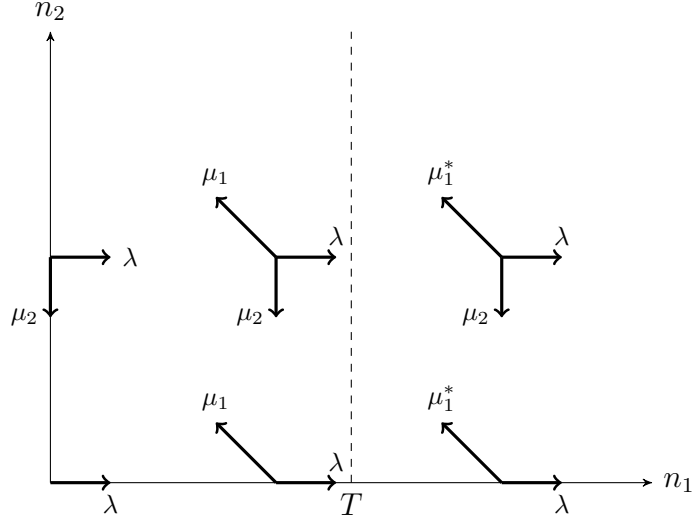
\begin{figure}[htb!]
\centering
\begin{tikzpicture}[scale = 1]
\draw [->, >=stealth'] (0,0) -- (0,6) node[above, thick] {$n_2$};
\draw [->, >=stealth'] (0,0) -- (8,0) node[right, thick] {$n_1$};
\draw [dashed] (4,0) node[below]{$T$} -- (4,6);

\draw [->, very thick] (3,3) -- ++(0.8,0) node[above] {\footnotesize $\lambda$};
\draw [->, very thick] (3,3) -- ++(-0.8,0.8) node[above] {\footnotesize $\mu_1$};
\draw [->, very thick] (3,3) -- ++(0,-0.8) node[left] {\footnotesize $\mu_2$}; 	

\draw [->, very thick] (6,3) -- ++(0.8,0) node[above] {\footnotesize $\lambda$};
\draw [->, very thick] (6,3) -- ++(-0.8,0.8) node[above] {\footnotesize $\mu_1^*$};
\draw [->, very thick] (6,3) -- ++(0,-0.8) node[left] {\footnotesize $\mu_2$}; 	


\draw [->, very thick] (3,0) -- ++(0.8,0) node[above] {\footnotesize $\lambda$};
\draw [->, very thick] (3,0) -- ++(-0.8,0.8) node[above] {\footnotesize $\mu_1$};


\draw [->, very thick] (6,0) -- ++(0.8,0) node[below] {\footnotesize $\lambda$};
\draw [->, very thick] (6,0) -- ++(-0.8,0.8) node[above] {\footnotesize $\mu_1^*$};


\draw [->, very thick] (0,3) -- ++(0.8,0) node[right] {\footnotesize $\lambda$};
\draw [->, very thick] (0,3) -- ++(0,-0.8) node[left] {\footnotesize $\mu_2$};


\draw [->, very thick] (0,0) -- ++(0.8,0) node[below] {\footnotesize $\lambda$};



\end{tikzpicture}
\caption{State space and transition rates of $\bar{R}$. }
\label{fig:extension_lp_tandem2dB_perturb}
\end{figure}

First, we consider the probability that an arriving job is rejected, \ie $F(n) = \mathbf{1}(n_1 = M)$. For instance, let $M = N$ and $T = 4$. Moreover, take for example $\lambda/\mu_1 = 1/2$, $\lambda/\mu_1^* = 1/3$ and $\lambda/\mu_2 = 1/3$. In Figure~\ref{fig:extension_lp_tandem2dB_block_probability_M}, we plot bounds on $\FF$ for various $M$. In addition, we plot the upper bound given by the comparison result in Problem~\ref{pr:extension_lp_comparison_upper}. The upper and lower bounds are denoted by $\FF_u$ and $\FF_l$ respectively, and the upper bound given by comparison result is denoted by $\FF^{(c)}_u$. Note that the y-axis is in logarithm scale. 

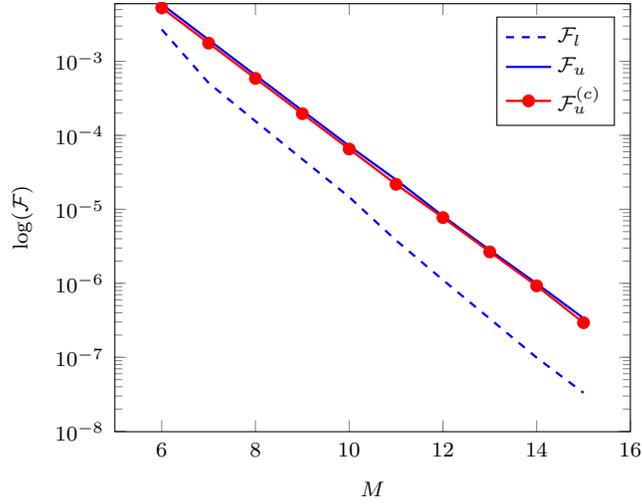
\begin{figure}[htb!]
\centering
\begin{tikzpicture}[scale = 1]
\begin{semilogyaxis}[
xlabel=$M$,ylabel=$\log(\FF)$, 
ymin = 0, ymax = 0.006,
xmin = 5,  
font=\scriptsize,
legend style={
	cells={anchor=west},
	legend pos=north east,
	font=\scriptsize,
}
]

\addplot[
mark=none,line width=.3mm,color=blue,dashed
]
table[
header=false,x index=0,y index=1,
]
{matlab/Extension_LP/tandem2dB_block_probability_M.csv};
\addlegendentry{$\FF_l$};

\addplot[
line width=.3mm,color=blue, mark=none
]
table[
header=false,x index=0,y index=2,
]
{matlab/Extension_LP/tandem2dB_block_probability_M.csv};
\addlegendentry{$\FF_u$};	

\addplot[
line width=.3mm,color=red, mark=*
]
table[
header=false,x index=0,y index=3,
]
{matlab/Extension_LP/tandem2dB_block_probability_M.csv};
\addlegendentry{$\FF^{(c)}_u$};
\end{semilogyaxis}
\end{tikzpicture}
\caption{Bounds on the rejecting probability for various $M$: $F(n) = \mathbf{1}(n_1 = M)$, $M = N$, $T = 4$, $\lambda/\mu_1 = 1/2$, $\lambda/\mu_1^* = 1/3$, $\lambda/\mu_2 = 1/3$.}
\label{fig:extension_lp_tandem2dB_block_probability_M}
\end{figure}


Next we consider the number of jobs in the system, \ie $F(n) = n_1 + n_2$. Again, we take $\lambda/\mu_1 = 1/2$, $\lambda/\mu_1^* = 1/3$ and $\lambda/\mu_2 = 1/3$. The bounds on $\FF$ as well as the comparison result are given in Figure~\ref{fig:extension_lp_tandem2dB_num_first_queue_M}. 

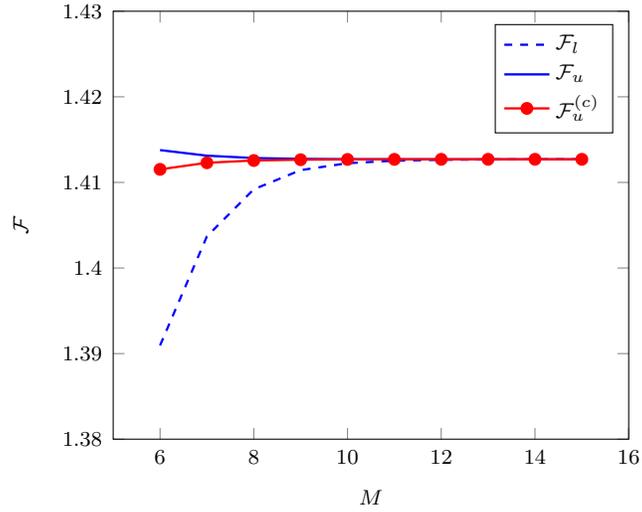
\begin{figure}[htb!]
\centering
\begin{tikzpicture}[scale = 1]
\begin{axis}[
xlabel=$M$,ylabel=$\FF$, 
ymin = 1.38, ymax = 1.43,
xmin = 5,  
font=\scriptsize,
legend style={
	cells={anchor=west},
	legend pos=north east,
	font=\scriptsize,
}
]

\addplot[
mark=none,line width=.3mm,color=blue, dashed
]
table[
header=false,x index=0,y index=1,
]
{matlab/Extension_LP/tandem2dB_num_first_queue_M.csv};
\addlegendentry{$\FF_l$};

\addplot[
line width=.3mm,color=blue, mark=none
]
table[
header=false,x index=0,y index=2,
]
{matlab/Extension_LP/tandem2dB_num_first_queue_M.csv};
\addlegendentry{$\FF_u$};	

\addplot[
line width=.3mm,color=red, mark=*
]
table[
header=false,x index=0,y index=3,
]
{matlab/Extension_LP/tandem2dB_num_first_queue_M.csv};
\addlegendentry{$\FF^{(c)}_u$};
\end{axis}
\end{tikzpicture}
\caption{Bounds on the number of jobs for various $M$: $F(n) = n_1 + n_2$, $M = N$, $T= 4$, $\lambda/\mu_1 = 1/2$, $\lambda/\mu_1^* = 1/3$, $\lambda/\mu_2 = 1/3$.}
\label{fig:extension_lp_tandem2dB_num_first_queue_M}
\end{figure}

At last we consider a different setting for the case $F(n)=n_1+n_2$. More precisely, we fix $M = N = 10$ and $T = 4$. Let $\lambda/\mu_1^* = 1/3$, $\lambda/\mu_2 = 1/2$ and consider various values for $\lambda/\mu_1$. Bounds on $\FF$ are given in Figure~\ref{fig:extension_lp_tandem2dB_num_first_queue_load}. 

%
%
%

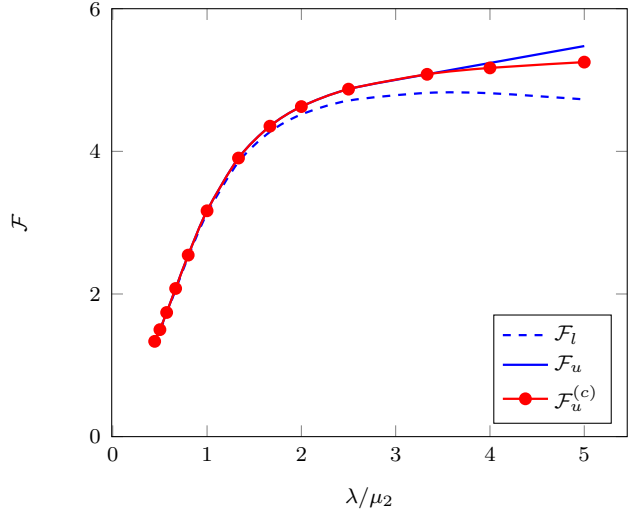
\begin{figure}[htb!]
\centering
\begin{tikzpicture}[scale = 1]
\begin{axis}[
xlabel=$\lambda/\mu_2$,ylabel=$\FF$, 
ymin = 0, ymax = 6,  
font=\scriptsize,
legend style={
	cells={anchor=west},
	legend pos=south east,
	font=\scriptsize,
}
]

\addplot[
mark=none,line width=.3mm,color=blue,dashed, smooth
]
table[
header=false,x index=0,y index=1,
]
{matlab/Extension_LP/tandem2dB_num_first_queue_load.csv};
\addlegendentry{$\FF_l$};

\addplot[
line width=.3mm,color=blue, mark=none, smooth
]
table[
header=false,x index=0,y index=2,
]
{matlab/Extension_LP/tandem2dB_num_first_queue_load.csv};
\addlegendentry{$\FF_u$};	

\addplot[
line width=.3mm,color=red, mark=*, smooth
]
table[
header=false,x index=0,y index=3,
]
{matlab/Extension_LP/tandem2dB_num_first_queue_load.csv};
\addlegendentry{$\FF^{(c)}_u$};
\end{axis}
\end{tikzpicture}
\caption{Bounds on the number of jobs for various $\lambda/\mu$: $F(n) = n_1+n_2$, $M = N = 10$, $T = 4$, $\lambda/\mu_1^* = 1/2$, $\lambda/\mu_2 = 1/3$.}
\label{fig:extension_lp_tandem2dB_num_first_queue_load}
\end{figure}

\subsection{Three-node coupled queue}

Consider a discrete-time coupled queue model in the three-dimensional space, \ie $S = \left\{ 0,1,\dots \right\}^3$. We restrict our attention to the symmetric case, \ie the service rates at all the nodes are $\mu$ when they are not empty. For the first queue, when the other two queues are empty, the service rate changes to $\mu^*$. The transition probabilities of $R$ are 
\begin{align}
P(n,n+e_1) =&\ P(n,n+e_2) = P(n,n+e_3) = \lambda, \\
P(n,n-e_2) =&\ \mathbf{1}(n-e_2 \in S)\mu, \quad P(n,n-e_3) = \mathbf{1}(n-e_3 \in S)\mu, \\
P(n,n-e_1) =& 
\begin{cases}
\mathbf{1}(n-e_1 \in S) \mu^*, & \textrm{if } n_2 = n_3 = 0, \\
\mathbf{1}(n-e_1 \in S) \mu, & \textrm{otherwise},
\end{cases} \\ 
P(n,n) =&\ 1 - \sum_{i=1}^3\sum_{u\in\left\{ e_i,-e_{i} \right\}} P(n,n+u),
\end{align}
where $e_1 = (1,0,0)$, $e_2 = (0,1,0)$ and $e_3 = (0,0,1)$. Without loss of generality, assume that $3\lambda+2\mu+\max\left\{ \mu,\mu^* \right\} \le 1$. Moreover, assume that $\lambda/\mu < 1$ for stability. 

As perturbed random walk $\bar{R}$, use
\begin{align}
\bar{P}(n,n+e_1) =&\ \bar{P}(n,n+e_2) = \bar{P}(n,n+e_3) = \lambda, \\
\bar{P}(n,n-e_2) =&\ \mathbf{1}(n-e_2 \in S)\mu, \quad \bar{P}(n,n-e_3) = \mathbf{1}(n-e_3 \in S)\mu, \\
\bar{P}(n,n-e_1) =&\ \mathbf{1}(n-e_1 \in S) \mu, \\
P(n,n) =&\ 1 - \sum_{i=1}^3\sum_{u\in\left\{ e_i,-e_{i} \right\}} \bar{P}(n,n+u).
\end{align}
Thus,
\begin{align}
\bar{\pi}(n) = (1-\rho)^3 \cdot \rho^{n_1+n_2+n_3},
\end{align}
where $\rho = \lambda/\mu$. Let $\mu^* = 1.5\mu$ and first consider the probability that the system is empty. The upper and lower bounds, together with the comparison result are given below in Figure~\ref{fig:extension_lp_couple3d_empty_load_1.5}. 

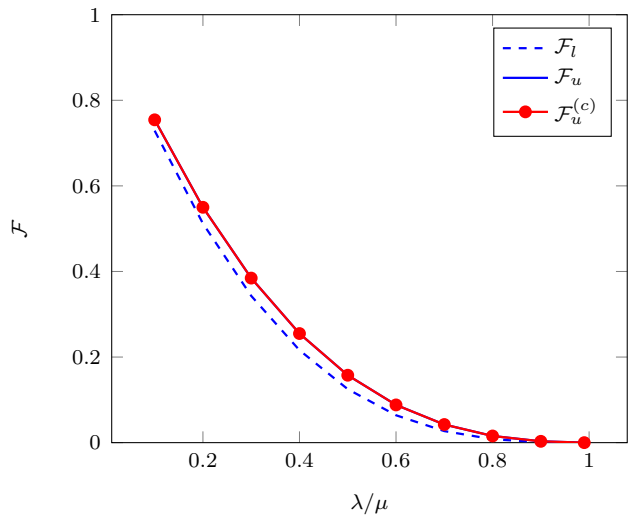
\begin{figure}[!ht]
\centering
\begin{tikzpicture}[scale=1]
\begin{axis}[
xlabel=$\lambda/\mu $,ylabel=$\FF$, 
ymin = 0, ymax = 1,
font=\scriptsize,
legend style={
	cells={anchor=west},
	legend pos=north east,
	font=\scriptsize,
}
]

\addplot[
mark=none,line width=.3mm,color=blue,dashed
]
table[
header=false,x index=0,y index=1,
]
{matlab/Extension_LP/couple3d_empty_load_1.5_new.csv};
\addlegendentry{$\FF_l$};
\addplot[
line width=.3mm,color=blue,
mark=none
]
table[
header=false,x index=0,y index=2,
]
{matlab/Extension_LP/couple3d_empty_load_1.5_new.csv};
\addlegendentry{$\FF_u$};	

\addplot[
line width=.3mm,color=red,
mark=*
]
table[
header=false,x index=0,y index=3,
]
{matlab/Extension_LP/couple3d_empty_load_1.5_new.csv};
\addlegendentry{$\FF^{(c)}_u$};
\end{axis} 
\end{tikzpicture}
\caption{Bounds on $\FF$ for various $\lambda/\mu$: $F(n)=\mathbf{1}(n=\mathbf{0})$, $\mu^*=1.5\mu$. }
\label{fig:extension_lp_couple3d_empty_load_1.5}
\end{figure}

From Figure~\ref{fig:extension_lp_couple3d_empty_load_1.5}, we see that the upper and lower bounds are very tight. Moreover, the comparison result is the same as the upper bound. 
Next, the performance measure $F(n) = n_1$ is considered. 

\begin{figure}[!ht]
\centering
\begin{tikzpicture}[scale=1]
\begin{axis}[
xlabel=$\lambda/\mu $,ylabel=$\FF$, 
ymin = 0, ymax = 10,
font=\scriptsize,
legend style={
	cells={anchor=west},
	legend pos=north west,
	font=\scriptsize,
}
]

\addplot[
mark=none,line width=.3mm,color=blue,dashed
]
table[
header=false,x index=0,y index=1,
]
{matlab/Extension_LP/couple3d_size1_load_1.5_new.csv};
\addlegendentry{$\FF_l$};
\addplot[
line width=.3mm,color=blue,
mark=none
]
table[
header=false,x index=0,y index=2,
]
{matlab/Extension_LP/couple3d_size1_load_1.5_new.csv};
\addlegendentry{$\FF_u$};	

\addplot[
line width=.3mm,color=red,
mark=*
]
table[
header=false,x index=0,y index=3,
]
{matlab/Extension_LP/couple3d_size1_load_1.5_new.csv};
\addlegendentry{$\FF^{(c)}_u$};

\end{axis}     
\end{tikzpicture}
\caption{Bounds on $\FF$ for various $\lambda/\mu$: $F(n)=n_1$, $\mu^*=1.5\mu$. }
\label{fig:extension_lp_couple3d_size1_load_1.5}
\end{figure}
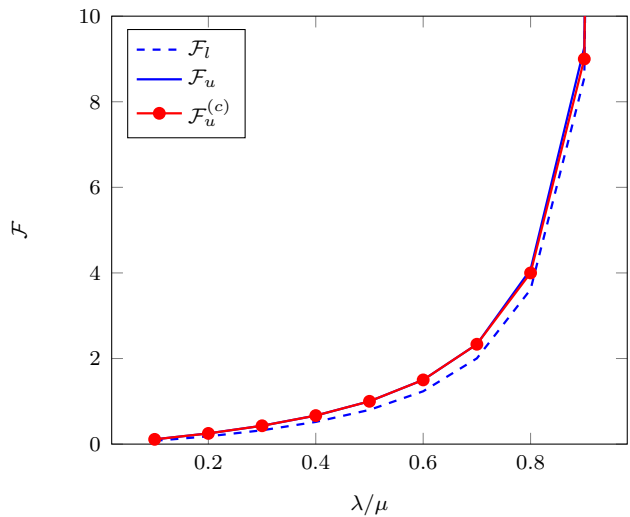 

From Figure~\ref{fig:extension_lp_couple3d_size1_load_1.5}, we see that the upper and lower bounds provide very good approximation to the performance for all the load between $0$ and $1$. In addition, the average number of jobs in the first queue increases as the load increases. 


\subsection{Three-node tandem system with boundary speed-up or slow-down}
\label{ssec:extension_lp_tandem3d}

Consider a tandem system containing three nodes. Every job arrives at node 1 and goes through all the nodes to complete its service. In the end, the job leaves the system through node 3. Let $\lambda$ be the arrival rate. Moreover, we assume that each server has the service rate $\mu$ when there are jobs in the queue. 
For server 1, the service rate changes to $\mu^*$, if both queue 2 and queue 3 become empty. 
Let $\mu^* = \eta\cdot\mu$. For the stability of the system, assume that $\lambda/\mu < 1$. 

\subsubsection*{The original random walk}

In this example, we have $S = \left\{ 0,1,\dots \right\}^3$ and the minimal partition $\minpar$ defined in Section~\ref{sec:introduction_model_description}. Notice that the tandem system described above is a continuous-time system. Therefore, we use the uniformization method to transform the continuous-time tandem system into a discrete-time $R$. Without loss of generality, assume that $\lambda + \max\{\mu,\mu^*\} + 2\mu \le 1$. Hence, we take the uniformization constant $1$. Then, the transition probabilities of the discrete-time $R$ are given below. 
\begin{align}
P(n,n+e_1) =&\ \lambda, \qquad P(n,n+d_2) = \mathbf{1}\left( n+d_2\in S \right)\mu, \\
P(n,n-e_3) =&\ \mathbf{1}\left( n-e_3\in S \right)\mu, \\
P(n,n+d_1) =& 
\begin{cases}
\mu^*, &\ \textrm{if } n_2=n_3=0, \\
\mu, &\ \textrm{otherwise}, 
\end{cases} \\
P(n,n) =&\ 1-\sum_{u\in  \left\{ e_1,d_1,d_2,d_3 \right\}} P(n,n+u), 
\end{align}
for all $n\in S$, with $e_1 = (1,0,0)$, $d_1 = (-1,1,0)$, $d_2 = (0,-1,1)$ and $e_3 = (0,0,1)$. 

\subsubsection*{The perturbed random walk}

For the perturbed random walks $\bar{R}$, we take
\begin{align}
\bar{P}(n,n+e_1) =&\ \lambda, \qquad \bar{P}(n,n+d_2) = \mathbf{1}\left( n+d_2\in S \right)\mu, \\
\bar{P}(n,n-e_3) =&\ \mathbf{1}\left( n-e_3\in S \right)\mu, \qquad \bar{P}(n,n+d_1) = \mathbf{1}\left( n+d_1\in S \right)\mu, \\
\bar{P}(n,n) = &\ 1-\sum_{u\in  \left\{ e_1,d_1,d_2,d_3 \right\}} \bar{P}(n,n+u). 
\end{align}
We know from~\cite{jackson1957networks} that the stationary distribution of $\bar{R}$ is,
\begin{align}
\bar{\pi}(n) = (1-\rho)^3 \cdot \rho^{n_1+n_2+n_3},
\end{align}
where $\rho = \lambda/\mu$. 

As performance measure, first we consider the probability that the system is empty, \ie $F(n) = \textbf{1}(n=\mathbf{0})$. In Figure~\ref{fig:extension_lp_tandem3d_empty_perturbation1}, we consider various values for $\eta$. In addition to the upper and lower bounds, we also include the comparison result given by Problem~\ref{pr:extension_lp_comparison_upper}, which provides an upper bound. 

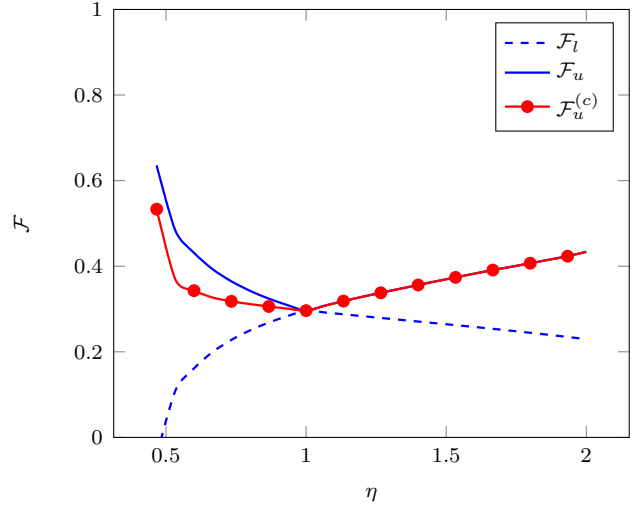
\begin{figure}[!htb]
\centering
\begin{tikzpicture}[scale=1]
\begin{axis}[
xlabel=$\eta$,ylabel=$\FF$, 
ymin = 0, ymax = 1,
font=\scriptsize,
legend style={
	cells={anchor=west},
	legend pos=north east,
	font=\scriptsize,
}
]
	
\addplot[
mark=none,line width=.3mm,color=blue,dashed, smooth
]
table[
header=false,x index=0,y index=1,
]
{matlab/Extension_LP/tandem3d_empty_perturbationerror1_new.csv};
\addlegendentry{$\FF_l$};

\addplot[
line width=.3mm,color=blue,
mark=none, smooth
]
table[
header=false,x index=0,y index=2,
]
{matlab/Extension_LP/tandem3d_empty_perturbationerror1_new.csv};
\addlegendentry{$\FF_u$};	

\addplot[
line width=.3mm,color=red,
mark=*, mark repeat = 2, smooth
]
table[
header=false,x index=0,y index=3,
]
{matlab/Extension_LP/tandem3d_empty_perturbationerror1_new.csv};
\addlegendentry{$\FF^{(c)}_u$};

\end{axis}     
\end{tikzpicture}
\caption{Bounds on $\FF$ for various $\eta$: $F = \mathbf{1}(n = \mathbf{0})$. }
\label{fig:extension_lp_tandem3d_empty_perturbation1}
\end{figure}  

The solid line is for the upper bounds $\FF_u$ and the dashed one is for lower bounds $\FF_l$. Moreover, the comparison upper bound is denoted by $\FF_u^{(c)}$. From Figure~\ref{fig:extension_lp_tandem3d_empty_perturbation1} we observe that the larger perturbation we make, the bigger the difference is between the upper and lower bounds. Moreover, we also see that the comparison result can give a better upper bound, when $\eta < 1$. 

Next, fix $\eta = 1.5$ and consider various values of $\lambda/\mu$. In Figure~\ref{fig:extension_lp_tandem3d_empty_load_1.5}, the upper and lower bounds as well as the comparison result are given on the probability that the system is empty. When $\lambda/\mu \ge 0.7$, the lower bound obtained by the optimization problem is negative. Hence, we use the trivial bound $0$. 

\begin{figure}[!htb]
\centering
\begin{tikzpicture}[scale=1]
\begin{axis}[
xlabel=$\lambda/\mu $,ylabel=$\FF$, 
ymin = 0, ymax = 1,
font=\scriptsize,
legend style={
	cells={anchor=west},
	legend pos=north east,
	font=\scriptsize,
}
]
	
\addplot[
mark=none,line width=.3mm,color=blue,dashed
]
table[
header=false,x index=0,y index=1,
]
{matlab/Extension_LP/tandem3d_empty_load_1.5_new.csv};
\addlegendentry{$\FF_l$};

\addplot[
line width=.3mm,color=blue,	
mark=none
]
table[
header=false,x index=0,y index=2,
]
{matlab/Extension_LP/tandem3d_empty_load_1.5_new.csv};
\addlegendentry{$\FF_u$};	

\addplot[
line width=.3mm,color=red,
mark=*
]
table[
header=false,x index=0,y index=3,
]
{matlab/Extension_LP/tandem3d_empty_load_1.5_new.csv};
\addlegendentry{$\FF^{(c)}_u$};
	
\end{axis}     
\end{tikzpicture}
\caption{Bounds on $\FF$ for various $\lambda/\mu$: $F(n)=\mathbf{1}(n=\mathbf{0})$, $\mu^* = 1.5\mu$. }
\label{fig:extension_lp_tandem3d_empty_load_1.5}
\end{figure}
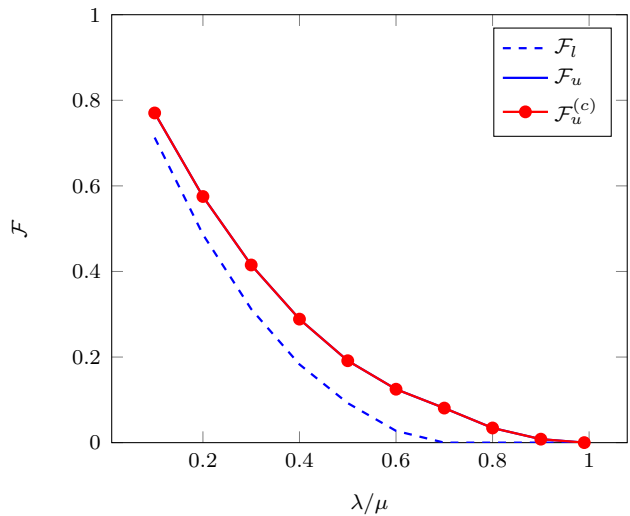

From Figure~\ref{fig:extension_lp_tandem3d_empty_load_1.5}, first we can notice that the upper and lower bounds are relatively tight. In addition, the comparison result is always the same as the upper bound, which is consistent with the result in Figure~\ref{fig:extension_lp_tandem3d_empty_perturbation1}, \ie the comparison result and the upper bound are the same for $\eta > 1$. Moreover, we see that as $\lambda/\mu$ increases, the probability that the system is empty decreases since the system becomes busier. Although at some point, the lower bound drops to $0$, the bounds given by our optimization problems are still reasonably good. 

Next, we consider a different performance measure, the average number of jobs in the first queue, \ie $F(n) = n_1$. Remark that the job in the server is also included for the number of jobs in the queue. 
Again, we fix $\eta = 1.5$ and consider various values of $\lambda/\mu$. The bounds and comparison result are given in Figure~\ref{fig:extension_lp_tandem3d_size1_load_1.5}. When the load is larger than $0.75$, the problems for both upper and lower bound are infeasible. Hence, the results for these cases are not included. 

\begin{figure}[!ht]
\centering
\begin{tikzpicture}[scale=1]
\begin{axis}[
xlabel=$\lambda/\mu $,ylabel=$f$, 
ymin = 0, ymax = 4,
font=\scriptsize,
legend style={
	cells={anchor=west},
	legend pos=north west,
	font=\scriptsize,
}
]

\addplot[ 
mark=none,line width=.3mm,color=blue,dashed
]
table[
header=false,x index=0,y index=1,
]
{matlab/Extension_LP/tandem3d_size1_load_1.5_new.csv};
\addlegendentry{$\FF_l$};
\addplot[
line width=.3mm,color=blue,
mark=none
]
table[
header=false,x index=0,y index=2,
]
{matlab/Extension_LP/tandem3d_size1_load_1.5_new.csv};
\addlegendentry{$\FF_u$};	

\addplot[
line width=.3mm,color=red,
mark=*
]
table[
header=false,x index=0,y index=3,
]
{matlab/Extension_LP/tandem3d_size1_load_1.5_new.csv};
\addlegendentry{$\FF^{(c)}_u$};

\end{axis}     
\end{tikzpicture}
\caption{Bounds on $\FF$ for various $\lambda/\mu$: $F(n)=n_1$, $\mu^*=1.5\mu$. }
\label{fig:extension_lp_tandem3d_size1_load_1.5}
\end{figure}
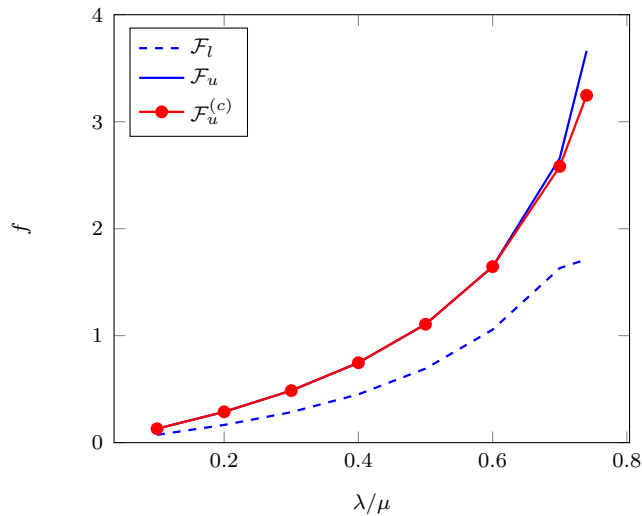


\section{Conclusions and discussion}
\label{sec:extension_lp_conclusion}

In this paper, we have considered random walks in $M$-dimensional positive orthant. Given a non-negative $C$-linear function, we have formulated optimization problems that provide upper and lower bounds on the stationary performance measure. Moreover, we have shown that these optimization problems can be reduced to linear programs with a finite number of variables and constraints. 

We have built up a numerical script in Python to implement the linear programs. In the paper we have used this script to obtain numerical bounds on stationary performance measures for various random walks. Through numerical experiments, we see that the linear programs for upper and lower bounds are not always feasible. In particular, once the load exceeds some threshold the problems often become infeasible and cannot return any bounds. The reason for this is still not known yet. One possible direction for future research is to use duality theory to find out exactly which constraints are violated. Then, we can understand more about how the linear programs work and then find a way to deal with the cases for heavy loads. Another interesting direction is to explore how to choose the objective function of Problem~\ref{pr:extension_lp_flow_v1} such that it improves the error bound. For instance, we may use a weighted sum as the objective function. It is also of interest to apply this numerical script to models where $M > 3$.

\bibliographystyle{IEEEtran}
\bibliography{references}

\appendix

\section{Proof of Theorem~\ref{thm:Markov_reward_approach_result}}
\label{sec:introduction_proof_MRA}

In this proof, we use the following shorthand notation for any $A:S \to [0,\infty)$, $B:S \to [0,\infty)$ and $C: S \times S \to [0,1]$,
\begin{align}
A \cdot B = \sum_{n\in S} A(n)B(n), \qquad A\cdot C(n) = C\cdot A(n) = \sum_{n^\prime \in S} A(n^\prime)C(n,n^\prime). 
\end{align}
From Equation~\eqref{eq:cumulative_reward}, we have
\begin{align}
\nonumber
\bar{F}^{t} - F^{t} =&\ (\bar{F} - F) + (\bar{P} \bar{F}^{t-1} - P F^{t-1}) \\
=&\ (\bar{F} - F) + (\bar{P} - P) F^{t-1} + \bar{P}(\bar{F}^{t-1} - F^{t-1}). 
\end{align}
Then, we use the relation above again for $\bar{F}^{t-1} - F^{t-1}$ in the RHS. Hence,
\begin{align}
\nonumber
\bar{F}^{t} - F^{t} =&\ (\bar{F} - F) + (\bar{P} - P) F^{t-1} \\
& \qquad\qquad + \bar{P} \left[ (\bar{F} - F) + (\bar{P} - P) F^{t-2} + \bar{P}(\bar{F}^{t-2} - F^{t-2}) \right] \\
=&\ \dots = \sum_{k=0}^{t-1} \left[ \bar{P}^k(\bar{F}-F) + \bar{P}^k\left( \bar{P}-P \right)F^{t-k-1} \right] + \bar{P}^{t+1}(\bar{F}^0 - F^0). 
\end{align}
The last item vanishes since $\bar{F}^0(n) = F^0(n) = 0$, for $n\in S$. Then, we have
\begin{align}
\label{eq:appendix_error_without_abs}
\bar{\pi}\cdot(\bar{F}^t - F^t) = \sum_{k=0}^{t-1} \cdot \bar{\pi}\left[ \bar{P}^k(\bar{F}-F) + \bar{P}^k\left( \bar{P}-P \right)F^{t-k-1} \right].
\end{align}
Since $\bar{\pi}$ is the stationary distribution of $\bar{R}$, $\bar{\pi} \bar{P}^k = \bar{\pi}$ for any $k \ge 0$. Therefore, taking the absolute value on both sides of~\eqref{eq:appendix_error_without_abs} we get
\begin{align}
\label{eq:appendix_error_abs}
\left| \bar{\pi}\cdot(\bar{F}^t - F^t) \right| = \left| \sum_{k=0}^{t-1}\sum_{n\in S} \bar{\pi}(n) \left\{ (\bar{F}(n)-F(n)) + \left[\left( \bar{P}-P \right)F^{t-k-1}\right](n) \right\} \right|
\end{align}
Moreover, by summing over $n^\prime \neq n$ and $n^\prime = n$ separately in the RHS of~\eqref{eq:appendix_error_abs}, we have
\begin{align}
\nonumber
\left[\left( \bar{P}-P \right)F^{t-k-1}\right](n) =& \sum_{n^\prime \neq n} \left[ \bar{P}(n,n^\prime)-P(n,n\prime) \right] F^{t-k-1}(n^\prime) \\
& - \sum_{n^\prime\neq n} \left[ \bar{P}(n,n^\prime)-P(n,n^\prime) \right]F^{t-k-1}(n) \\
\label{eq:difference_cumu_reward}
=& \sum_{n^\prime\in S} \left[ \bar{P}(n,n^\prime)-P(n,n^\prime) \right]D^{t-k-1}(n,n^\prime).
\end{align}
Therefore, by~\eqref{eq:difference_cumu_reward} and~\eqref{eq:error_bound_condition} we have
\begin{align}
\nonumber
& \left| \bar{\pi}\cdot(\bar{F}^t - F^t) \right| \\
\nonumber
\le& \sum_{k=0}^{t-1}\sum_{n\in S}\bar{\pi}(n) \left| \bar{F}(n)-F(n) + \sum_{n^\prime\in S} \left[ \bar{P}(n,n^\prime)-P(n,n^\prime) \right]D^{t-k-1}(n,n^\prime) \right| \\
\nonumber
\le & \sum_{n\in S}\bar{\pi}(n) tG(n) \\
=&\ t\bar{\pi}\cdot G.
\end{align}
For any $n\in S$, we know that
\begin{eqnarray*}
\FF = \lim_{t\to \infty}\frac{F^t(n)}{t}, \qquad \bar{\FF} = \lim_{t\to \infty}\frac{\bar{F}(n)}{t}.
\end{eqnarray*}
Remark that the equation above also holds when $R$ or $\tilde{R}$ has a single absorbing communicating class which is a subset of $S$, while the other states are transient. Therefore,
\begin{align}
\lim_{t\to \infty} \frac{\left| \bar{\pi}\cdot(\bar{F}^t - F^t) \right|}{t} \le \bar{\pi}\cdot G \quad \Longleftrightarrow \quad \left| \bar{\FF} - \FF \right| \le \bar{\pi}\cdot G.
\end{align}

\end{document}